\documentclass[a4paper,12pt]{amsart}

\usepackage[
  margin=30mm,
  marginparwidth=25mm,     
  marginparsep=2mm,       
  bottom=25mm,
  ]{geometry}

\usepackage[latin1]{inputenc}
\usepackage[T1]{fontenc}
\usepackage[english]{babel}
\usepackage{amsmath,amssymb,amsthm,mathtools}
\usepackage[makeroom]{cancel}
\expandafter\let\expandafter\xbf\csname bfseries \endcsname
\expandafter\let\expandafter\xmd\csname mdseries \endcsname
\let\xbar\bar
\let\bar\xbar
\expandafter\let\csname bfseries \endcsname\xbf
\expandafter\let\csname mdseries \endcsname\xmd

\usepackage{latexsym}
\usepackage{delarray}
\usepackage{bbm}
\usepackage{scrextend}
\usepackage{hyperref}
\usepackage[pdftex,usenames,dvipsnames]{xcolor}
\usepackage{datetime}
\usepackage{mathrsfs}
\usepackage{enumitem}
\usepackage{tikz}

\usepackage{MnSymbol,wasysym}

\usepackage{multicol}

\newtheorem{theorem}{Theorem}[section]
\newtheorem{proposition}{Proposition}[section]

\newtheorem{corollary}{Corollary}[section]

\newtheorem*{Rem*}{Remark}

\newcommand{\N}{\mathbb{N}}

\newcommand{\C}{\mathbb{C}}
\newcommand{\Z}{\mathbb{Z}}

\newcommand{\OO}{\mathcal{O}}
\newcommand{\V}{\mathcal{V}}

\newcommand{\edproof}{ $\hfill {\Box}$}

\allowdisplaybreaks 

\title
[Variation and oscillation Jacobi]
{Variation and oscillation for semigroups associated with discrete Jacobi operators}

\author[J. J. Betancor]{J. J. Betancor}
\address{Jorge J. Betancor\newline
	Departamento de An\'alisis Matem\'atico, Universidad de La Laguna,\newline
	Campus de Anchieta, Avda. Astrof\'isico S\'anchez, s/n,\newline
	38721 La Laguna (Sta. Cruz de Tenerife), Spain}
\email{jbetanco@ull.es}

\author[M. De Le\'on-Contreras]{M. De Le\'on-Contreras$^*$}
\address{\newline
       Marta De Le\'on-Contreras \newline
       Departamento de An\'alisis Matem\'atico, Universidad de La Laguna,\newline
       Campus de Anchieta, Avda. Astrof\'isico S\'anchez, s/n,\newline
       38721 La Laguna (Sta. Cruz de Tenerife), Spain}
\email{mleoncon@ull.edu.es}



\thanks{$^*$Corresponding author}
\begin{document}

\begin{abstract}
In this paper we prove weighted $\ell^p$-inequalities for variation and oscillation operators defined by semigroups of operators associated with discrete Jacobi operators. Also, we establish that certain maximal operators involving sums of differences  of discrete Jacobi  semigroups are bounded on weighted $\ell^p$-spaces. $\ell^p$-boundedness properties for the considered  operators  provide information about the convergence of the semigroup of operators defining them.
\end{abstract}
\maketitle

\setcounter{secnumdepth}{3}
\setcounter{tocdepth}{3}


\section{Introduction}

The $\rho$-variational inequalities for bounded martingales were {first} studied by L\'epingle in \cite{Leb}. These properties can be seen as extensions of Doob's maximal inequality and they give quantitative versions of the martingale convergence theorem. Generalizations of L\'epingle's results can be found in \cite{Bou, MSZ} and \cite{PX}.

Bourgain (\cite{Bou}) was the first in studying variational inequalities in ergodic theory. He rediscovered L\'epingle's inequality and used it to establish pointwise convergence of ergodic averages involving polynomial orbits. The seminal paper \cite{Bou} opened the study of variational inequalities in harmonic analysis and ergodic theory (\cite{CJRW1, CJRW2, JKRW, JRe, JSW, JW,MSS, MST} and \cite{MSZ}. Oscillation and variation estimates for semigroups of operators can be found,  for instance, in \cite{BCT,HMMT,JW,RZ} and \cite{TZ2}.

Let $\rho>0$ and $\{a_t\}_{t>0}\subset \C$. We define the $\rho$-variation of $\{a_t\}_{t>0}$, $\V_\rho(\{a_t\}_{t>0})$, by 
$$
\V_\rho(\{a_t\}_{t>0})=\sup_{\substack{0<t_n<t_{n-1}<\dots<t_1\\n\in\N}}\left(\sum_{{j=1}}^{n-1}|a_{t_j}-a_{t_{j+1}}|^\rho \right)^{1/\rho}.
$$

Let $\{t_j\}_{j\in\N}\subset (0,\infty)$ be a decreasing sequence such that $t_j\to 0$, as $j\to \infty$. The oscillation of $\{a_t\}_{t>0}$,  $\OO(\{a_t\}_{t>0},\{t_j\}_{j\in\N})$, is defined by
$$
\OO(\{a_t\}_{t>0},\{t_j\}_{j\in\N})=\left(\sum_{{j=1}}^{\infty}\sup_{t_{j+1}\le \epsilon_{j+1}<\epsilon_j\le t_{j}}|a_{\epsilon_j}-a_{\epsilon_{j+1}}|^2 \right)^{1/2}.
$$
Let $\lambda>0.$  We define the $\lambda$-jump of  $\{a_t\}_{t>0}$, $\Lambda(\{a_t\}_{t>0},\lambda)$ by
\begin{align*}
    \Lambda(\{a_t\}_{t>0},\lambda)=\sup\{n\in\N: \:\exists \: s_1<t_1\le s_2<t_2\le\dots\le s_n<t_n, \text{ such that }\\ |a_{t_i}-a_{s_i}|>\lambda, \: i=1,\dots,n\}.
\end{align*}
Variations, oscillation and jumps provide us information about convergence properties for $\{a_t\}_{t>0}.$

 Suppose that $\{T_t\}_{t>0}$ is a family of operators in $L^p(X,\mu)$ with $1\le p<\infty$, where $(X,\mu)$ is a measure space. We define, for every $f\in L^p(X,\mu)$, 
 \begin{align*}
 &\V_\rho(\{T_t\}_{t>0})(f)(x):=\V_\rho(\{T_t(f)(x)\}_{t>0}),\\
 &\OO(\{T_t\}_{t>0},\{t_j\}_{j\in\N})(f)(x):=\OO(\{T_t(f)(x)\}_{t>0},\{t_j\}_{j\in\N})\\
 \text{and}\\
 & \Lambda(\{T_t\}_{t>0},\lambda)(f)(x):= \Lambda(\{T_t(f)(x)\}_{t>0},\lambda).
 \end{align*}
 An important issue in this point is the measurability of these new functions. Comments about this property can be encountered after \cite[Theorem 1.2]{CJRW1}. Our objective is to get $L^p$-boundedness properties for the variations, oscillation and jump operators. As usual, in order to obtain $L^p$-boundedness for the $\rho$-variation operator, we need to consider $\rho>2$. This is the case when we work with martingales, see \cite{JW} and \cite{Q}. The oscillation operator, which has exponent 2, can be a good substitute of the $2-$variation operator. According to \cite[(1.15)]{MSS}, we can see uniform $\lambda$-jump estimates as endpoint estimates for $\rho$-variations, $\rho>2$. Moreover,  it is proved in \cite[Theorem 1.9]{MSS} that {the oscillation operator} cannot be interpreted as an endpoint in the sense of inequality \cite[(1.15)]{MSS} for $\rho$-variations, $\rho>2$.
 
 Let $\{a_j\}_{j\in\Z}$ be an increasing sequence in $(0,\infty)$ and $\{b_j\}_{j\in\Z}$  a bounded real sequence. According to {\cite{BLMMTT}} and \cite{JRo}, we define, for every $N=(N_1,N_2)$ with $N_1,N_2\in\Z$, $N_1<N_2$, the operator $S_N$ by
 $$
S_{\{a_j\}_{j\in{\Z}},N}^{\{b_j\}_{j\in{\Z}}}(\{T_t\}_{t>0})(f)=\sum_{j=N_1}^{N_2}b_j(T_{a_{j+1}}f-T_{a_{j}}f),
 $$
 and the corresponding maximal operator, $S_*$, by
 $$
S_{\{a_j\}_{j\in{\Z},*}}^{\{b_j\}_{j\in{\Z}}}(\{T_t\}_{t>0})(f)=\sup_{\substack{N=(N_1,N_2)\\
 	N_1,N_2\in\Z,\: N_1<N_2}}\left|S_{\{a_j\}_{j\in{\Z},N}}^{\{b_j\}_{j\in{\Z}}}(\{T_t\}_{t>0})(f)\right|.
 $$

These operators can help us to complete the picture of the convergence properties of $\{T_t\}_{t>0}$. By \cite[Remark 1]{JRo}, we need to assume that  the sequence  $\{a_j\}_{j\in\Z}$ satisfies some extra condition (lacunarity, for instance)  in order to obtain $L^p$-boundedness properties for the operator $S_*$. 

Our objective is to establish $L^p$-inequalities for all above operators when $\{T_t\}_{t>0}$ is the discrete Jacobi heat semigroup.

We now recall some definitions and properties about Jacobi polynomials that we will use along the paper.

Let $\alpha,\beta>-1$. For every $n\in{\N_0:=\N\cup\{0\}}$, we define the $n-$th Jacobi polynomial $P_n^{(\alpha,\beta)}$ by
	$$
	P_n^{(\alpha,\beta)}(x)=\frac{(-1)^n}{2^n n!}(1-x)^{-\alpha}(1+x)^{-\beta} \frac{d^n}{dx^n}((1-x)^{\alpha+n}(1+x)^{\beta+n}), \quad x\in (-1,1),
	$$
	see \cite[p.67, formula (4.3.1)]{Sz}.
	
	We also consider $p_n^{(\alpha,\beta)}=w_n^{(\alpha,\beta)}P_n^{(\alpha,\beta)}$, $n\in {\N_0}$, where 
	
	$$
	w_n^{(\alpha,\beta)}=\sqrt{\frac{(2n+\alpha+\beta+1)\Gamma(n+1)\Gamma(n+\alpha+\beta+1)}{2^{\alpha+\beta+1}\Gamma(n+\alpha+1)\Gamma(n+\beta+1)}}, \quad n\in\N,
	$$
	and
	$$
		w_0^{(\alpha,\beta)}=\sqrt{\frac{\Gamma(\alpha+\beta+2)}{2^{\alpha+\beta+1}\Gamma(\alpha+1)\Gamma(\beta+1)}}.
	$$
	The sequence $\{p_n^{(\alpha,\beta)}\}_{n\in{\N_0}}$ is an orthonormal basis in $L^2((-1,1),\mu_{\alpha,\beta})$, where $d\mu_{\alpha,\beta}(x)=(1-x)^{\alpha}(1+x)^{\beta}dx.$

	We define the difference operator $J^{(\alpha,\beta)}$ as follows,
	
	$$
		J^{(\alpha,\beta)}(f)(n)=a_{n-1}^{(\alpha,\beta)}f(n-1)+b_n^{(\alpha,\beta)}f(n)+a_n^{(\alpha,\beta)}f(n+1),\quad n\in\N,\,
	$$
	and
	$$
	J^{(\alpha,\beta)}(f)(0)=b_0^{(\alpha,\beta)}f(0)+a_0^{(\alpha,\beta)}f(1),
	$$
	where 
	\begin{align*}
	a_n^{(\alpha,\beta)}&=\frac{2}{2n+\alpha+\beta+2}\sqrt{\frac{(n+1)(n+\alpha+1)(n+\beta+1)(n+\alpha+\beta+1)}{(2n+\alpha+\beta+1)(2n+\alpha+\beta+3)}}, \quad n\in\N,\\\
	a_0^{(\alpha,\beta)}&= \frac{2}{\alpha+\beta+2}\sqrt{\frac{(\alpha+1)(\beta+1)}{\alpha+\beta+3}},\\
	b_n^{(\alpha,\beta)}&=\frac{\beta^2-\alpha^2}{(2n+\alpha+\beta)(2n+\alpha+\beta+2)}-1, \quad n\in\N,\ \\
	\text{and}\\
	b_0^{(\alpha,\beta)}&={-\frac{2\alpha+2}{\alpha+\beta+2}.}
		\end{align*}
	The spectrum of the operator $	J^{(\alpha,\beta)}$ is $[-2,0]$ and, for every $x\in[-1,1]$,
	$$
		J^{(\alpha,\beta)}p_n(x)=(x-1)p_n(x), \quad n\in{\N_0}.
	$$

The operator $	J^{(\alpha,\beta)}$ is bounded from ${\ell^p(\N_0)}$ into itself, for every $1\le p\le \infty$. Furthermore, the operator  $	J^{(\alpha,\beta)}$  is selfadjoint on  ${\ell^2(\N_0)}$ and $-	J^{(\alpha,\beta)}$ is a positive operator in
${\ell^2(\N_0)}$. We denote by $\{W_t^{(\alpha,\beta)}\}_{t>0}:=\{e^{tJ^{(\alpha,\beta)}}\}_{t>0}$ the semigroup of operators generated by $J^{(\alpha,\beta)}$.

We define the $(\alpha,\beta)$-Fourier transform as follows
$$
\mathcal{F}^{(\alpha,\beta)}(f)=\sum_{n=0}^\infty f(n)p_n^{(\alpha,\beta)}, \quad f\in{\ell^2(\N_0)}.
$$
Thus, $\mathcal{F}^{(\alpha,\beta)}$ is an isometry from ${\ell^2(\N_0)}$ into $L^2((-1,1),\mu_{\alpha,\beta})$.

We can write, for every $t>0$,
$$
W_t^{(\alpha,\beta)}(f)(n)=\int_{-1}^1e^{-t(1-x)}\mathcal{F}^{(\alpha,\beta)}(f)(x)p_n^{(\alpha,\beta)}(x) d\mu_{\alpha,\beta}(x), \quad n\in{\N_0}.
$$
We can see that, for every $t>0$,
$$
W_t^{(\alpha,\beta)}(f)(n)=\sum_{m=0}^\infty f(m) K_t^{(\alpha,\beta)}(n,m),\quad n\in{\N_0},
$$
where
\begin{equation}\label{eq2}
K_t^{(\alpha,\beta)}(n,m)=\int_{-1}^1 e^{-t(1-x)}p_n^{(\alpha,\beta)}(x)p_m^{(\alpha,\beta)}(x)d\mu_{\alpha,\beta}(x), \quad n,m\in{\N_0}.
\end{equation}
Gasper (\cite{Ga3,Ga1} and \cite{Ga2}) established the linearisation property for the product of Jacobi polynomials and his results can be transfered to the polynomials $\{p_n^{(\alpha,\beta)}\}_{n\in{\N_0}}$. Then, a convolution operator can be defined in the $\{p_n^{(\alpha,\beta)}\}_{n\in{\N_0}}$ that is transformed by $\mathcal{F}^{(\alpha,\beta)}$  in the pointwise product. For every $t>0$, $W_t^{(\alpha,\beta)}$ can be seen as a convolution operator.

Askey (\cite{A}) proved a power weighted transplantation theorem for Jacobi coefficients. Recently, Arenas, Ciaurri and Labarga (\cite{ACL4}) extended Askey's result by considering the transplantation operator as a singular integral and weights in the Muckenhoupt class for $({\N_0},\mathcal{P}({\N_0}),\mu_d)$. By taking as inspiration point the study of Ciaurri, Gillespie, Roncal, Torrea and Varona (\cite{CGRTV}) about harmonic analysis operators associated with the discrete Laplacian, Betancor, Castro, Fari\~na and Rodr\'iguez-Mesa (\cite{BCFR}) established weighted $L^p$-inequalities for harmonic analysis operators in the discrete ultraspherical setting. They took advantage of the discrete convolution operator associated with the ultraspherical polynomials in the discrete context (\cite{Hi}). Jacobi polynomials reduce to ultraspherical polynomials when $\alpha=\beta$. Arenas, Ciaurri and Labarga (\cite{ACL3,ACL1, ACL2}) extended the results in \cite{BCFR} to the Jacobi context. They needed to use a different procedure from the one employed in \cite{BCFR} for the ultraspherical setting because they can not use the convolution operator. Also, as in \cite{BCFR}  and \cite{CGRTV}, scalar and vector-valued Calder\'on-Zygmund theory for singular integrals  was a main tool. Maximal operators and Littlewood-Paley functions defined for the heat semigroup $\{ W_t^{(\alpha,\beta)}\}_{t>0}$ were studied in \cite{ACL1} and \cite{ACL3}, respectively.

Riesz transforms associated with the discrete Jacobi operator $J^{(\alpha,\beta)}$ were considered in \cite{ACL2}.

We now state our results. A real sequence $\{v_n\}_{n\in{\N_0}}$ is said to be a weight when $v_n>0$, $n\in{\N_0}$. If $1<p<\infty$, we say that a weight $\{v_n\}_{n\in{\N_0}}$ is in $A_p({\N_0})$ when
$$
\sup_{\substack{0\le n\le m\\n,m\in{\N_0}}}\frac{1}{(m-n+1)^p}\sum_{k=n}^mv_k\left(\sum_{k=n}^mv_k^{\frac{-1}{p-1}} \right)^{{p-1}}<\infty.
$$
A weight $\{v_n\}_{n\in{\N_0}}$ belongs to the class $A_1({\N_0})$ when
$$
\sup_{\substack{0\le n\le m\\n,m\in{\N_0}}}\frac{1}{m-n+1}\left(\sum_{k=n}^mv_k \right)\max_{n\le k\le m}\frac{1}{v_k}<\infty.
$$

For every weight $v$ on ${\N_0}$ and $1\le p<\infty$, we denote by $\ell^p({\N_0},w)$ the weighted $p$-Lebesgue space on $({\N_0},\mathcal{P}({\N_0}),\mu_d)$ and by $\ell^{1,\infty}({\N_0},w)$ the $(1,\infty)$-weighted Lorentz space on $({\N_0},\mathcal{P}({\N_0}),\mu_d)$.

\begin{theorem}\label{teo1.1}
Let $\alpha,\beta\ge -\frac{1}{2}$, $\rho>2$ and  $\{t_j\}_{j\in\N}$ be a decreasing sequence in $(0,\infty)$ {that converges to $0$}.
\begin{itemize}
\item[$a)$] The variation operator $\V_\rho(\{ W_t^{(\alpha,\beta)}\}_{t>0})$ and the oscillation operator \newline $\OO(\{W_t^{(\alpha,\beta)}\}_{t>0},\{t_j\}_{j\in\N})$ are bounded from  $\ell^p({\N_0},v)$ into itself, for every $1<p<\infty$ and $v\in A_p({\N_0})$, and from $\ell^1({\N_0},v)$ into $\ell^{1,\infty}({\N_0},v)$, for every $v\in A_1({\N_0})$.

\item[$b)$] The family $\{{\lambda}(\Lambda(\{W_t^{(\alpha,\beta)}\}_{t>0},\lambda))^{1/\rho}\}_{\lambda>0}$, is uniformly bounded from $\ell^p({\N_0},v)$ into itself, for every $1<p<\infty$ and $v\in A_p({\N_0})$, and from $\ell^1({\N_0},v)$ into $\ell^{1,\infty}({\N_0},v)$, for every $v\in A_1({\N_0})$.

\end{itemize}

\end{theorem}
Results in Theorem \ref{teo1.1} had not been established for the semigroups generated by the discrete Laplacian and the ultraspherical operators. Now the results in the ultraspherical setting can be deduced from Theorem \ref{teo1.1}  when $\alpha=\beta$. Moreover, it will be explained in Section \ref{sec2} that our procedure in the proof of Theorem \ref{teo1.1}  allows us to prove the corresponding results for the semigroup generated by the discrete Laplacian.

Calder\'on-Zygmund theory for vector-valued singular integrals (\cite{RRT1} and \cite{RRT2}) will be a main tool in our proof of Theorem \ref{teo1.1}. We can not use the transplantation theorem as in \cite{ACL3} because, in contrast with the Littlewood-Paley functions, variation and oscillation operators are not related with Hilbert norms. We need to refine the arguments developed in \cite{ACL1} by using asymptotics for Jacobi polynomials and Bessel functions.

\begin{theorem}\label{teo1.2}
	Let $\alpha,\beta\ge -\frac{1}{2}$. Assume that  $\{a_j\}_{j\in\Z}$ is a $\rho$-lacunary sequence in $(0,\infty)$ with $\rho>1$ and  $\{b_j\}_{j\in\Z}$ is a bounded sequence of real numbers. The maximal operator $S_{\{a_j\}_{j\in{\Z}},*}^{\{b_j\}_{j\in{\Z}}}(\{W_t^{(\alpha,\beta)}\}_{t>0})$ is bounded from $\ell^p({\N_0},w)$ into itself, for every $1<p<\infty$ and $w\in A_p({\N_0})$, and from $\ell^1({\N_0},w)$ into $\ell^{1,\infty}({\N_0},w)$, for every $w\in A_1({\N_0})$.
	\end{theorem}
	Ben Salem (\cite{Be}) solved an initial value problem associated with a fractional diffusion equation involving fractional powers of the Jacobi operator,   {$(J^{(\alpha,\beta)})^\gamma$}, and Caputo fractional derivatives in time. By using subordination, from  Theorems \ref{teo1.1} and \ref{teo1.2} we can deduce the corresponding results when  $\{W_t^{(\alpha,\beta)}\}_{t>0}$ is replaced by the semigroup of operators generated by {$(J^{(\alpha,\beta)})^\gamma$}, $\gamma>0$. 

In the next section we prove Theorems \ref{teo1.1} and \ref{teo1.2}. Throughout this paper, we will always denote by $C$ and $c$ positive constants that can change in each occurrence.

\section{Proof of Theorem  \ref{teo1.1}  }\label{sec2}

\subsection{Proof of Theorem  \ref{teo1.1}  for $\V_\rho(\{ W_t^{(\alpha,\beta)}\}_{t>0})$}\label{subsec2.1}\textcolor{white}{}\newline

First, we shall prove that  $\V_\rho(\{ W_t^{(\alpha,\beta)}\}_{t>0})$ is bounded from ${\ell^2({\N_0})}$ into itself.

We have that $J^{(\alpha,\beta)}p_n^{(\alpha,\beta)}(x)=(x-1) p_n^{(\alpha,\beta)}(x)$, $x\in (-1,1)$ and $n\in{\N_0}$. Hence,  $J^{(\alpha,\beta)}p_n^{(\alpha,\beta)}(1)=0$, $n\in{\N_0}$. We consider the operator $\tilde{J}^{(\alpha,\beta)}$ defined by
$$
\tilde{J}^{(\alpha,\beta)}(f)(n)=\frac{1}{p_n^{(\alpha,\beta)}(1)}J^{(\alpha,\beta)}(p_{\cdot}^{(\alpha,\beta)}f)(n), \;\:n\in{\N_0},
$$
and the weight $v^{(\alpha,\beta)}=\{p_n^{(\alpha,\beta)}(1)\}_{n\in{\N_0}}$.

Let $t>0$. We define the operator $\tilde{W}_t^{(\alpha,\beta)}$ on $\ell^p({\N_0}, v^{(\alpha,\beta)})$, $1\le p\le\infty$ by
$$
\tilde{W}_t^{(\alpha,\beta)}(f)(n)=\frac{1}{p_n^{(\alpha,\beta)}(1)}{W}_t^{(\alpha,\beta)}(p_{\cdot}^{(\alpha,\beta)}f)(n),\;\:n\in{\N_0}.
$$
We can write, for every $f\in\ell^p({\N_0},v^{(\alpha,\beta)})$, $1\le p<\infty$, 
$$
\tilde{W}_t^{(\alpha,\beta)}(f)(n)=\sum_{m=0}^\infty f(m)\tilde{K}_t^{(\alpha,\beta)}(n,m)(p_m^{(\alpha,\beta)}(1))^2, \:\;n\in{\N_0},
$$
where 
$$
\tilde{K}_t^{(\alpha,\beta)}(n,m)=\frac{{K}_t^{(\alpha,\beta)}(n,m)}{p_n^{(\alpha,\beta)}(1)p_m^{(\alpha,\beta)}(1)},\:\;n,m\in{\N_0}.
$$
Since $\alpha\ge\beta\ge -1/2$, see \cite[Theorem 1]{Ga2}, according to \cite[Theorem {3.2}]{ACL1}, we have that $K_t^{(\alpha,\beta)}(n,m)\ge 0$ and therefore  $\tilde{K}_t^{(\alpha,\beta)}(n,m)\ge 0$, $n,m\in{\N_0}$.

The family $\{\tilde{W}_s^{(\alpha,\beta)}\}_{s>0}$ is the semigroup of operators generated by $\tilde{J}^{(\alpha,\beta)}$ in $\ell^p({\N_0},v^{(\alpha,\beta)})$, $1\le p\le\infty$. Since ${J}^{(\alpha,\beta)}p_n^{(\alpha,\beta)}(1)=0,$ $n\in{\N_0}$, we deduce that  $\tilde{W}_s^{(\alpha,\beta)}(1)(n)=1$, $n\in{\N_0}$, that is, the semigroup $\{\tilde{W}_s^{(\alpha,\beta)}\}_{s>0}$ is Markovian. Furthermore, by using Jensen inequality we deduce that
$$
|\tilde{W}_t^{(\alpha,\beta)}(f)(n)|^p\le \sum_{m=0}^\infty \tilde{K}_t^{(\alpha,\beta)}(n,m)(p_m^{(\alpha,\beta)}(1))^2|f(m)|^p, \quad n\in{\N_0} \text{ and } t>0,
$$
for every $1\le p<\infty$. Since $\tilde{K}_t^{(\alpha,\beta)}(n,m)=\tilde{K}_t^{(\alpha,\beta)}(m,n)$, $n,m\in{\N_0}$, it follows that $\tilde{W}_t^{(\alpha,\beta)}$ is a contraction in $\ell^p({\N_0},v^{(\alpha,\beta)})$, for every $1\le p\le \infty$, and it is selfadjoint on $\ell^2({\N_0},v^{(\alpha,\beta)})$.

We have proved that  $\{\tilde{W}_t^{(\alpha,\beta)}\}_{t>0}$  is a diffusion semigroup in the Stein's sense (\cite{StLP}).

According to \cite[Corollary 4.5]{LX} (see also \cite[{Theorem 3.3}]{JRe}) we have that the $\rho$-variation operator  $\V_\rho(\{ \tilde{W}_t^{(\alpha,\beta)}\}_{t>0})$ is bounded from  $\ell^p({\N_0},v^{(\alpha,\beta)})$ into itself, for every $1< p<\infty$. By taking into account that
$$
\V_\rho(\{ \tilde{W}_t^{(\alpha,\beta)}\}_{t>0})(f)(n)=\frac{1}{p_n^{(\alpha,\beta)}(1)}\V_\rho(\{ {W}_t^{(\alpha,\beta)}\}_{t>0})(p_\cdot^{(\alpha,\beta)}(1)(f))(n), \quad n\in{\N_0},
$$
we deduce that  $\V_\rho(\{ {W}_t^{(\alpha,\beta)}\}_{t>0})$ is  bounded from  ${\ell^2(\N_0)}$ into itself.

Now we shall use Calder\'on-Zygmund theory for vector-valued singular integrals (see \cite[Theorem 2.1]{BCFR}). If $g$ is a complex-valued function defined on $(0,\infty)$, we define 
$$
\|g\|_\rho=\sup_{\substack{0<t_n<t_{n-1}<\cdots<t_1\\n\in\N}}\left(\sum_{j=1}^{n-1}|g(t_j)-g(t_{j+1})|^\rho \right)^{1/\rho},
$$
and the linear space $E_\rho$ that consists of all those $g:(0,\infty)\to\C$ such that $\|g\|_\rho<\infty.$ It is clear that  $\|g\|_\rho=0$ if, and only if, $g$ is constant. By identifying those functions that differ in a constant, $\|\cdot\|_\rho$ is a norm in  $E_\rho$ and $( E_\rho,\|\cdot\|_\rho)$ is a Banach space.

We can write
$$
\V_\rho(\{ {W}_t^{(\alpha,\beta)}\}_{t>0})(f)(n)=\|  {W}_t^{(\alpha,\beta)}(f)(n)\|_\rho, \quad n\in{\N_0}.
$$

 $\|\cdot\|_\rho$ is not a Hilbert norm. Then, a transplantation theorem can not be applied, in contrast with the case of Littlewood-Paley functions considered in \cite{ACL3}. 
 
 We are going to see that 
 \begin{equation}\label{eq2.1}
 \| {K}_t^{(\alpha,\beta)}(n,m)\|_\rho\le \frac{C}{|n-m|}, \quad n,m\in{\N_0}, \: n\neq m,
 \end{equation}
 and
 \begin{equation}\label{eq2.2}
 \|{K}_t^{(\alpha,\beta)}(n,m)-{K}_t^{(\alpha,\beta)}(l,m) \|_\rho\le C\frac{|n-l|}{|n-m|^2}, \quad |n-m|>2|n-l|,\: \frac{m}{2}\le n,l\le \frac{3m}{2}.
 \end{equation}
 
 First, we prove \eqref{eq2.1}. According to \cite[Lemma 5.1]{ACL1}, we have that
 \begin{align*}
 {K}_t^{(\alpha,\beta)}(n,m)&=w_n^{(\alpha,\beta)}w_m^{(\alpha,\beta)}\frac{(n+\alpha+\beta+1)(m+\alpha+\beta+1)}{2(n-m)(n+m+\alpha+\beta+1)}t\Bigg( \frac{1}{m+\alpha+\beta+1} H_t^{(\alpha,\beta)}(n,m)\\
 &\qquad-\frac{1}{n+\alpha+\beta+1}H_t^{(\alpha,\beta)}(m,n)\Bigg),\quad n,m\in{\N},\:\; n\neq m \text{ and } t>0,
 \end{align*}
where, for $k,l\in\N, \:k\ge 1 \text{ and } t>0$,
$$
H_t^{(\alpha,\beta)}(k,l)=\int_{-1}^1e^{-t(1-x)}P_{k-1}^{(\alpha+1,\beta+1)}(x)P_{l}^{(\alpha,\beta)}(x)(1-x)^{\alpha+1}(1+x)^{\beta+1}dx.
$$

Since $w_n^{(\alpha,\beta)}\sim \sqrt{n}$, $n\in\N$, in order to prove \eqref{eq2.1} when $n,m\in{\N}, $ $n\neq m$, it is sufficient to see that
$$
\|tH_t^{(\alpha,\beta)}(n,m)\|_\rho\le\frac{C}{\sqrt{nm}},\quad  n,m\in{\N},\:\; n\neq m. 
$$

Let $n,m\in{\N},\:\; n\neq m. $ We decompose
$$
H_t^{(\alpha,\beta)}(n,m)=H_{t,1}^{(\alpha,\beta)}(n,m)+H_{t,2}^{(\alpha,\beta)}(n,m),\quad t>0,
$$
where
$$
H_{t,1}^{(\alpha,\beta)}(n,m)=\int_0^1e^{-t(1-x)}P_{n-1}^{(\alpha+1,\beta+1)}(x)P_{m}^{(\alpha,\beta)}(x)(1-x)^{\alpha+1}(1+x)^{\beta+1}dx,\quad\; t>0.
$$
Suppose that $g:(0,\infty)\to \C$ is a differentiable function. We can write
\begin{align}\label{eq2.3}
\|g\|_\rho&=\sup_{\substack{0<t_n<t_{n-1}<\cdots<t_1\\n\in\N}}\left(\sum_{j=1}^{n-1}|g(t_j)-g(t_{j+1})|^\rho \right)^{1/\rho}\nonumber\\
&\le \sup_{\substack{0<t_n<t_{n-1}<\cdots<t_1\\n\in\N}}\left(\sum_{j=1}^{n-1}\left|\int_{t_{j+1}}^{t_{j}}g'(t)\;dt\right|^\rho \right)^{1/\rho}\nonumber\\
&\le \sup_{\substack{0<t_n<t_{n-1}<\cdots<t_1\\n\in\N}}\sum_{j=1}^{n-1}\left|\int_{t_{j+1}}^{t_{j}}g'(t)\;dt\right| 
\le \int_{0}^\infty|g'(t)|\;dt. 
\end{align}
We will use \eqref{eq2.3} several times in the sequel.

According to \cite[(7.32.6)]{Sz}, we have that
\begin{equation}\label{eq2.4}
|P_k^{(\alpha,\beta)}(x)|\le \frac{C}{\sqrt{k}}(1-x)^{-\alpha/2-1/4}(1+x)^{-\beta/2-1/4},\quad x\in(-1,1) \text{ and } k\in{\N}. 
\end{equation}
By using \eqref{eq2.3} and \eqref{eq2.4}, we get 
\begin{align}\label{eq2.5}
\|tH_{t,2}^{(\alpha,\beta)}(n,m)\|_\rho&\le\int_{0}^\infty\left|\frac{d}{dt}\left(tH_{t,2}^{(\alpha,\beta)}(n,m)\right) \right|dt\nonumber\\
&\le \frac{C}{\sqrt{nm}}\int_0^\infty \int_{-1}^0 e^{-t(1-x)}(t(1-x)+1)dxdt\le \frac{C}{\sqrt{nm}}.
\end{align}
On the other hand, since $P_0^{(\alpha+1,\beta+1)}(x)=1$, $x\in(-1,1)$, it follows that
$$
H_{t,1}^{(\alpha,\beta)}(1,m)=\int_0^1e^{-t(1-x)}P_{m}^{(\alpha,\beta)}(x)(1-x)^{\alpha+1}(1+x)^{\beta+1}dx,\quad t>0.
$$
Then, \eqref{eq2.4} leads to
\begin{align*}
\|tH_{t,1}^{(\alpha,\beta)}(1,m)\|_\rho&\le \frac{C}{\sqrt{m}}\int_0^\infty \int_0^1 e^{-t(1-x)}(t(1-x)+1)(1-x)^{{\alpha/2+}3/4}(1+x)^{{\beta/2+}3/4}\;dxdt\\
&\le \frac{C}{\sqrt{m}}.
\end{align*}
In \cite[Theorem 8.21.12]{Sz}, it was established that
\begin{align}\label{eq2.6}
\left(\sin\frac{\theta}{2}\right)^\alpha\left(\cos\frac{\theta}{2}\right)^\beta P_l^{(\alpha,\beta)}(\cos\theta)&=\gamma_l^{-\alpha}\frac{\Gamma(l+\alpha+1)}{\Gamma(l+1)}\left(\frac{\theta}{\sin\theta}\right)^{1/2}J_\alpha(\gamma_l\theta)\nonumber\nonumber\\
&\quad+\left\{\begin{array}{ll}
\theta^{1/2}O(l^{-3/2}),\quad \frac{c}{l}\le \theta\le n-\epsilon,\\
\theta^{\alpha+2}O(l^{\alpha}),\quad 0< \theta<\frac{c}{l},
\end{array}\right. \quad l\in{\N},
\end{align}
where $\gamma_l=l+\frac{\alpha+\beta+1}{2}$. Here, $c$ and $\epsilon$ are fixed positive numbers. By \cite[{(5.16.1)}]{Leb} we have that
\begin{align}\label{eq2.7}
J_\alpha(z)\le C\left\{\begin{array}{ll} 
z^\alpha,\quad 0<z<1,\\
z^{-1/2},\quad z\ge 1.
\end{array}\right.
\end{align}
We define 
\begin{align*}
F_l^{(\alpha,\beta)}(\theta)=P_l^{(\alpha,\beta)}(\cos\theta)-\gamma_l^{-\alpha}\frac{\Gamma(l+\alpha+1)}{\Gamma(l+1)}\left(\sin\frac{\theta}{2}\right)^{-\alpha}\left(\cos\frac{\theta}{2}\right)^{-\beta}\left(\frac{\theta}{\sin\theta}\right)^{1/2}&J_\alpha(\gamma_l\theta),\\ \quad\theta\in\left(0,\frac{\pi}{2}\right) \text{ and } l\in{\N}. 
\end{align*}
Assume now that $n>1$. By performing the change of variables $x=\cos \theta$, we can write
\begin{align*}
&H_{t,1}^{(\alpha,\beta)}(n,m)={2^{\alpha+\beta+3}}\int_0^{\frac{\pi}{2}}e^{-t(1-\cos\theta)}P_{n-1}^{(\alpha+1,\beta+1)}(\cos\theta)P_{m}^{(\alpha,\beta)}(\cos\theta)\left(\sin\frac{\theta}{2}\right)^{2\alpha+3}\left(\cos\frac{\theta}{2}\right)^{2\beta+3}d\theta\\
&={2^{\alpha+\beta+3}}\Bigg[\int_0^{\frac{\pi}{2}}e^{-t(1-\cos\theta)}F_{n-1}^{(\alpha+1,\beta+1)}(\theta)F_{m}^{(\alpha,\beta)}(\theta)\left(\sin\frac{\theta}{2}\right)^{2\alpha+3}\left(\cos\frac{\theta}{2}\right)^{2\beta+3}d\theta\\
&+\gamma_m^{-\alpha}\frac{\Gamma(m+\alpha+1)}{\Gamma(m+1)}\int_0^{\frac{\pi}{2}}e^{-t(1-\cos\theta)}F_{n-1}^{(\alpha+1,\beta+1)}(\theta)\left(\frac{\theta}{\sin\theta}\right)^{1/2}J_\alpha(\gamma_m\theta)\left(\sin\frac{\theta}{2}\right)^{\alpha+3}\left(\cos\frac{\theta}{2}\right)^{\beta+3}d\theta\\
&+\gamma_{n}^{-\alpha-1}\frac{\Gamma({n+\alpha})}{\Gamma(n)}\int_0^{\frac{\pi}{2}}e^{-t(1-\cos\theta)}F_{m}^{(\alpha,\beta)}(\theta)\left(\frac{\theta}{\sin\theta}\right)^{1/2}J_{\alpha+1}(\gamma_{n}\theta)\left(\sin\frac{\theta}{2}\right)^{\alpha+2}\left(\cos\frac{\theta}{2}\right)^{\beta+2}d\theta\\
&+\frac{\gamma_{n}^{-\alpha-1}\gamma_m^{-\alpha}\Gamma({n+\alpha})\Gamma(m+\alpha+1)}{{2}\Gamma(n)\Gamma(m+1)}\int_0^{\frac{\pi}{2}}e^{-t(1-\cos\theta)}\theta J_{\alpha+1}(\gamma_{n}\theta)J_{\alpha}(\gamma_m\theta)\sin\frac{\theta}{2}\cos\frac{\theta}{2}d\theta\Bigg]\\
&:=\sum_{j=1}^4 H_{t,1,j}^{(\alpha,\beta)}(n,m),\qquad t>0.
\end{align*}
Suppose that $m>n$. By \eqref{eq2.6} we get that
\begin{align*}
|\partial_t(tH_{t,1,1}^{(\alpha,\beta)}(n,m))|&\le Cn^{\alpha+1}m^\alpha\int_0^{\frac{1}{m}}e^{-ct\theta^2}(1+t\theta^2)\theta^{2\alpha+7}d\theta \\
&\quad+Cn^{\alpha+1}m^{-3/2}\int_{\frac{1}{m}}^{\frac{1}{n}}e^{-ct\theta^2}(1+t\theta^2)\theta^{\alpha+\frac{11}{2}}d\theta\\
&\quad+C(nm)^{-3/2}\int_{\frac{1}{n}}^{\frac{\pi}{2}}e^{-ct\theta^2}(1+t\theta^2)\theta^{3}d\theta, \qquad t>0.
\end{align*}
Then,
\begin{align*}
\int_0^\infty|\partial_t(tH_{t,1,1}^{(\alpha,\beta)}(n,m))|dt&\le Cn^{\alpha+1}m^\alpha\int_0^{\frac{1}{m}}\theta^{2\alpha+5}d\theta 
+Cn^{\alpha+1}m^{-3/2}\int_{\frac{1}{m}}^{\frac{1}{n}}\theta^{\alpha+\frac{7}{2}}d\theta\\
&\quad+C(nm)^{-3/2}\int_{\frac{1}{n}}^{\frac{\pi}{2}}\theta \;d\theta\\
&\le\frac{C}{(nm)^{3/2}}.
\end{align*}
Since $\gamma_k\sim k$, $k\in{\N}$, \eqref{eq2.6} and \eqref{eq2.7} lead to 
\begin{align*}
|\partial_t(tH_{t,1,2}^{(\alpha,\beta)}(n,m))|&\le Cn^{\alpha+1}m^\alpha\int_0^{\frac{1}{m}}e^{-ct\theta^2}(1+t\theta^2)\theta^{2\alpha+5}d\theta \\
&\quad+Cn^{\alpha+1}m^{-1/2}\int_{\frac{1}{m}}^{\frac{1}{n}}e^{-ct\theta^2}(1+t\theta^2)\theta^{\alpha+\frac{9}{2}}d\theta\\
&\quad+Cn^{-3/2}m^{-1/2}\int_{\frac{1}{n}}^{\frac{\pi}{2}}e^{-ct\theta^2}(1+t\theta^2)\theta^{2}d\theta, \qquad t>0.
\end{align*}
It follows that
\begin{align*}
\int_0^\infty|\partial_t(tH_{t,1,2}^{(\alpha,\beta)}(n,m))|dt&\le Cn^{\alpha+1}m^\alpha\int_0^{\frac{1}{m}}\theta^{2\alpha+3}d\theta 
+Cn^{\alpha+1}m^{-1/2}\int_{\frac{1}{m}}^{\frac{1}{n}}\theta^{\alpha+\frac{5}{2}}d\theta\\
&\quad+Cn^{-3/2}m^{-1/2}\int_{\frac{1}{n}}^{\frac{\pi}{2}}\;d\theta\\
&\le\frac{C}{n^{3/2}m^{1/2}}.
\end{align*}
Similarly, we obtain that 
\begin{align*}
\int_0^\infty|\partial_t(tH_{t,1,3}^{(\alpha,\beta)}(n,m))|dt
&\le\frac{C}{n^{1/2}m^{3/2}}.
\end{align*}
Thus, we conclude that 
$$
\sum_{j=1}^3\int_0^\infty|\partial_t(t H_{t,1,j}^{(\alpha,\beta)}(n,m))|dt\le\frac{C}{\sqrt{nm}}.
$$
We are going to see that

$$
\int_0^\infty|\partial_t(t Z_t^{\alpha}(n,m))|dt\le\frac{C}{\sqrt{nm}},
$$
where
$$
Z_t^{\alpha}(n,m)=\int_0^{\frac{\pi}{2}}e^{-t(1-\cos\theta)}\theta J_{\alpha+1}(\gamma_{n}\theta)J_{\alpha}(\gamma_m\theta)\sin\frac{\theta}{2}\cos\frac{\theta}{2}d\theta, \quad t>0.
$$
Again, since $\gamma_k\sim k$, $k\in{\N}$,   by using \eqref{eq2.7} we get
\begin{align*}
&\left| \partial_t\left(t Z_t^{\alpha}(n,m)-t\int_{\frac{1}{n}}^{\frac{\pi}{2}}e^{-t(1-\cos\theta)}\theta J_{\alpha+1}(\gamma_{n}\theta)J_{\alpha}(\gamma_m\theta)\sin\frac{\theta}{2}\cos\frac{\theta}{2}d\theta\right)\right|\\
&=\left| \partial_t\left(t\int_0^{\frac{1}{n}}e^{-t(1-\cos\theta)}\theta J_{\alpha+1}(\gamma_{n}\theta)J_{\alpha}(\gamma_m\theta)\sin\frac{\theta}{2}\cos\frac{\theta}{2}d\theta\right)\right|\\
&\le C\left( n^{\alpha+1}m^\alpha\int_0^{\frac{1}{m}}e^{-c\theta^2t}\theta^{2\alpha+3}d\theta+n^{\alpha+1}m^{-1/2}\int_{\frac{1}{m}}^{\frac{1}{n}}e^{-c\theta^2t}\theta^{\alpha+5/2}d\theta  \right), \quad t>0.
\end{align*}
Then,
\begin{align*}
\int_0^\infty&\left| \partial_t\left(t Z_t^{\alpha}(n,m)-t\int_{\frac{1}{n}}^{\frac{\pi}{2}}e^{-t(1-\cos\theta)}\theta J_{\alpha+1}(\gamma_{n}\theta)J_{\alpha}(\gamma_m\theta)\sin\frac{\theta}{2}\cos\frac{\theta}{2}d\theta\right)\right|dt\\
&\le C\left( n^{\alpha+1}m^\alpha\int_0^{\frac{1}{m}}\theta^{2\alpha+1}d\theta+n^{\alpha+1}m^{-1/2}\int_{\frac{1}{m}}^{\frac{1}{n}}\theta^{\alpha+1/2}d\theta  \right)\\
&\le C\left(\frac{n^{\alpha+1}m^\alpha}{m^{2\alpha+2}}+\frac{n^{\alpha+1}m^{-1/2}}{n^{\alpha+3/2}}\right)\le C\left(\frac{1}{m} +\frac{1}{\sqrt{nm}}\right)\le\frac{C}{\sqrt{nm}}.
\end{align*}
According to \cite[{(5.11.6)}]{Leb}, we have that
\begin{equation}\label{eq2.8}
J_\alpha(z)=\sqrt{\frac{2}{\pi z}}\cos\left(z-\frac{\alpha\pi}{2}-\frac{\pi}{4}\right)+g_\alpha(z), \quad z>0,
\end{equation}
where $|g_\alpha(z)|\le C z^{-3/2}, \quad z\ge 1.$

We define, 
$$
Q_t^\alpha(n,m)=\int_{\frac{1}{n}}^{\frac{\pi}{2}}e^{-t(1-\cos\theta)}\theta J_{\alpha+1}(\gamma_{n}\theta)J_{\alpha}(\gamma_m\theta)\sin\frac{\theta}{2}\cos\frac{\theta}{2}d\theta,\quad t>0.
$$

We can write
\begin{align*}
Q_t^\alpha(n,m)&=\frac{1}{\pi\sqrt{nm}}\int_{\frac{1}{n}}^{\frac{\pi}{2}}e^{-t(1-\cos\theta)} \cos\left(\gamma_{n}\theta-\frac{(\alpha+1)\pi}{2}-\frac{\pi}{4}\right)\cos\left(\gamma_m\theta-\frac{\alpha\pi}{2}-\frac{\pi}{4}\right)\sin\theta \;d\theta\\
&+\frac{1}{\sqrt{2\pi n}}\int_{\frac{1}{n}}^{\frac{\pi}{2}}e^{-t(1-\cos\theta)} \cos\left(\gamma_{n}\theta-\frac{(\alpha+1)\pi}{2}-\frac{\pi}{4}\right)g_\alpha(\gamma_m\theta)\sqrt{\theta}\sin\theta \;d\theta\\
&+\frac{1}{\sqrt{2\pi m}}\int_{\frac{1}{n}}^{\frac{\pi}{2}}e^{-t(1-\cos\theta)}g_{\alpha+1}(\gamma_{n}\theta) \cos\left(\gamma_m\theta-\frac{\alpha\pi}{2}-\frac{\pi}{4}\right)\sqrt{\theta}\sin\theta \;d\theta\\
&+{\frac{1}{2}}\int_{\frac{1}{n}}^{\frac{\pi}{2}}e^{-t(1-\cos\theta)}g_{\alpha+1}(\gamma_{n}\theta) g_\alpha(\gamma_m\theta){\theta}\sin\theta \;d\theta\\
&=\sum_{j=1}^4Q_{t,j}^\alpha(n,m),\quad t>0.
\end{align*}
By using \eqref{eq2.8}, we get
\begin{align*}
\sum_{j=2}^4\int_0^\infty |\partial_t(tQ_{t,j}^\alpha(n,m))|dt&\le C\Bigg(\frac{1}{n^{1/2}m^{3/2}} +\frac{1}{n^{3/2}m^{1/2}}\Bigg)\int_0^\infty\int_{\frac{1}{n}}^{\frac{\pi}{2}} e^{-ct\theta^2}d\theta dt\\&\quad+\frac{C}{(nm)^{3/2}}\int_0^\infty\int_{\frac{1}{n}}^{\frac{\pi}{2}} e^{-ct\theta^2}\frac{d\theta}{\theta^3} dt  \\&\le C\left(\frac{\sqrt{n}}{m^{3/2}}+\frac{1}{\sqrt{nm}}\right)\le \frac{C}{\sqrt{nm}}.
\end{align*}
Our next objective is to see that
$$
\int_0^\infty |\partial_t(tQ_{t,1}^\alpha(n,m))|dt\le\frac{C}{\sqrt{nm}}.
$$
A straightforward manipulation leads to
\begin{align}\label{eq2.9}
2\cos&\left(\gamma_{n}\theta-\frac{(\alpha+1)\pi}{2}-\frac{\pi}{4}\right)\cos\left(\gamma_m\theta-\frac{\alpha\pi}{2}-\frac{\pi}{4}\right)=2\sin\left(\gamma_{n}\theta-\eta\right)\cos\left(\gamma_m\theta-\eta\right)\nonumber\\
&=\sin\left((\gamma_{n}+\gamma_m)\theta-2\eta\right)+\sin\left((\gamma_{n}-\gamma_m)\theta\right)\nonumber\\
&=\cos(2\eta)(\sin((n+m)\theta)(\cos(\rho\theta)-1)+\sin((n+m)\theta)+\sin(\rho\theta)\cos((n+m)\theta))\nonumber\\&\:-\sin(2\eta)(\cos((n+m)\theta)(\cos(\rho\theta)-1)+\cos((n+m)\theta){-}\sin(\rho\theta)\sin((n+m)\theta))\nonumber\\&\:+\sin((n-m)\theta),
\end{align}
where $\eta=\frac{\alpha\pi}{2}+\frac{\pi}{4}$ and $\rho=\alpha+\beta+1$.

We consider
$$
R_t(n,m)=t\int_{\frac{1}{n}}^{\frac{\pi}{2}}e^{-t(1-\cos\theta)}\sin((n-m)\theta)\sin\theta d\theta, \quad t>0.
$$
We shall prove that
\begin{equation}\label{eq2.10}
\int_0^\infty|\partial_t R_t(n,m)|dt\le C.
\end{equation}
By partial integration we obtain that
\begin{align*}
R_t(n,m)&=-t\int_{\frac{1}{n}}^{\frac{\pi}{2}}e^{-t(1-\cos\theta)}\frac{d}{d\theta}\left(\frac{\cos((n-m)\theta)}{n-m}\right)\sin\theta d\theta\\&=t(S_{n,m}(t,\pi/2)-S_{n,m}(t,1/n)-\mathbb{R}_t(n,m)),\quad t>0,
\end{align*}
where
$$
S_{n,m}(t,\theta)=e^{-t(1-\cos\theta)}\frac{\cos((n-m)\theta)}{m-n}\sin\theta,\quad \theta\in \left(0,\frac{\pi}{2}\right) \text{ and } t>0,
$$ and
$$
\mathbb{R}_t(n,m)=\int_{\frac{1}{n}}^{\frac{\pi}{2}}e^{-t(1-\cos\theta)}\frac{\cos((n-m)\theta)}{m-n}(-t\sin^2\theta+\cos\theta) d\theta,\quad t>0.
$$
We have that
\begin{align*}
	\int_0^\infty|\partial_t[tS_{n,m}(t,\pi/2)]|dt\le \frac{C}{m-n}\int_0^\infty (1+t)e^{-t}dt\le \frac{C}{m-n}
	\end{align*}
	and
	\begin{align*}
	\int_0^\infty|\partial_t[tS_{n,m}(t,{1}/{n})]|dt&\le \frac{C}{m-n}\int_0^\infty \left(1+t\left(1-\cos\frac{1}{n}\right)\right)e^{-t\left(1-\cos\frac{1}{n}\right)}\sin\frac{1}{n}dt\\
	&\le \frac{C}{n(m-n)}\int_0^\infty e^{\frac{-ct}{n^2}}dt\le C\frac{n}{m-n}.
	\end{align*}
We also get
\begin{align*}
\int_0^\infty|\partial_t(t\mathbb{R}_t(n,m))|dt&\le\frac{C}{m-n}\int_0^\infty\int_{\frac{1}{n}}^{\frac{\pi}{2}}e^{-ct\theta^2} (t\theta^2+1+t^2\theta^4)d\theta dt\le \frac{C}{m-n}\int_{\frac{1}{n}}^{\frac{\pi}{2}}\frac{d\theta}{\theta^2}\\&=C\frac{n}{m-n}.
\end{align*}

We conclude that
\begin{equation*}
\int_0^\infty|\partial_t R_t(n,m)|dt\le C\frac{n}{m-n}\le C,
\end{equation*}
provided that $m>2n$.

By proceeding in a similar way we can see that
\begin{align*}
\int_0^\infty\Big|\partial_t\Big[t\int_{\frac{1}{n}}^{\frac{\pi}{2}}&e^{-t(1-\cos\theta)}\sin\theta\big(\cos(2\eta)[\sin((n+m)\theta)(\cos(\rho\theta)-1)+\sin((n+m)\theta)\\&+\sin(\rho\theta)\cos((n+m)\theta)]\nonumber-\sin(2\eta)[\cos((n+m)\theta)(\cos(\rho\theta)-1)\\&+\cos((n+m)\theta){-}\sin(\rho\theta)\sin((n+m)\theta)]\big)d\theta\Big]\Big|dt\le C.
\end{align*}
Note that the last inequality holds for every $n,m\in\N$.

Suppose that $1<m-n<n$. We decompose $R_t(n,m)$ as follows
\begin{align*}
R_t(n,m)&=t\int_{\frac{1}{n}}^{\frac{1}{m-n}}e^{-t(1-\cos\theta)}\sin((n-m)\theta)\sin\theta \;d\theta\\
&+t\int_{\frac{1}{m-n}}^{\frac{\pi}{2}}e^{-t(1-\cos\theta)}\sin((n-m)\theta)\sin\theta \;d\theta\\
&=R_t^1(n,m)+R_t^2(n,m),\quad t>0.
\end{align*}
We get
\begin{align*}
\int_0^\infty|\partial_tR_t^1(n,m)|dt&\le C\int_0^\infty\int_{\frac{1}{n}}^{\frac{1}{m-n}}e^{-ct\theta^2}(1+t\theta^2)(m-n)\theta^2d\theta dt\\
&\le C(m-n)\int_{\frac{1}{n}}^{\frac{1}{m-n}}d\theta\le C.
\end{align*}
On the other hand, by proceeding as in the proof of \eqref{eq2.10}, we can see that
$$
\int_0^\infty|\partial_tR_t^2(n,m)|dt\le C.
$$
We conclude that 
$$
\int_0^\infty|\partial_tR_t(n,m)|dt\le C.
$$
By combining all above estimates we prove that
$$
\|tH_{t}^{(\alpha,\beta)}(n,m)\|_\rho\le \frac{C}{\sqrt{nm}},\quad n,m\in{\N}, \: m>n.
$$
Also, the same arguments allow us to obtain that
$$
\|tH_{t}^{(\alpha,\beta)}(n,m)\|_\rho\le \frac{C}{\sqrt{nm}},\quad n,m\in{\N}, \: n>m.
$$
Thus, we have proved that
$$
\|K_{t}^{(\alpha,\beta)}(n,m)\|_\rho\le \frac{C}{|n-m|},\quad n,m\in{\N}, \: m\neq n.
$$
Let now $m\in{\N}.$ According to \cite[Lemma 5.1]{ACL1}, we have that
$$
K_{t}^{(\alpha,\beta)}(0,m)=w_0^{(\alpha,\beta)}w_m^{(\alpha,\beta)}\frac{t}{2m}\mathcal{H}_t^{(\alpha,\beta)}(m), \quad t>0,
$$
where
$$
\mathcal{H}_t^{(\alpha,\beta)}(m)=\int_{-1}^1 e^{-t(1-x)}P_{m-1}^{(\alpha+1,\beta+1)}(x)(1-x)^{\alpha+1}(1+x)^{\beta+1}dx,\quad t>0.
$$
By using \eqref{eq2.4}, we get
$$
|\partial_t[t\mathcal{H}_t^{(\alpha,\beta)}(m)]|\le\frac{C}{\sqrt{m}}\int_{-1}^1 e^{-t(1-x)}((1-x)t+1)(1-x)^{{\frac{\alpha}{2}+\frac{1}{4}}}(1+x)^{{\frac{\beta}{2}+\frac{1}{4}}}dx,\quad t>0.
$$
Then, since $w_k^{(\alpha,\beta)}\sim \sqrt{k}$, $k\in{\N}$, we obtain
$$
\| K_{t}^{(\alpha,\beta)}(0,m)\|_\rho\le\int_0^\infty|\partial_t[t\mathcal{H}_t^{(\alpha,\beta)}(m)]|dt\le \frac{C}{m}.
$$
Similarly, we get
$$
\| K_{t}^{(\alpha,\beta)}(m,0)\|_\rho\le \frac{C}{m}.
$$
Therefore, the proof of \eqref{eq2.1} is finished.

By proceeding as in \cite[pp. 13-14]{ACL1}, we can see that in order to prove \eqref{eq2.2}, it is sufficient to establish that
\begin{equation}\label{eq2.12}
\| K_{t}^{(\alpha,\beta)}(n+1,m)- K_{t}^{(\alpha,\beta)}(n,m) \|_\rho\le \frac{C}{|n-m|^2},
\end{equation}
for every $n,m\in\N$, $n\neq m$, $m/2\le n\le  3m/2$.

Suppose that $n,m\in{\N_0}$, $n\neq m$, $m/2\le n\le  3m/2$. Then, $n\neq 0\neq m$ and $m=2$ when $n=1$. Assume also that $(n,m)\neq(1,2)$.

By using \eqref{eq2.3} and the arguments in \cite[pp. 18-19]{ACL1} we can deduce that \eqref{eq2.12} holds once we will prove that
\begin{equation}\label{eq2.13}
\int_0^\infty |\partial_tD_t^{(\alpha,\beta)}(n,m)|dt\le \frac{C}{\sqrt{nm}\;|n-m|^2},
\end{equation}
where 
$$
D_t^{(\alpha,\beta)}(n,m)=\int_{-1}^1 e^{-t(1-x)}P_{n}^{(\alpha+1,\beta)}(x)P_{m}^{(\alpha,\beta)}(x)(1-x)^{\alpha+1}(1+x)^{\beta}dx,\quad t>0.
$$

According to \cite[Lemma 5.1 (a)]{ACL1}, we get
\begin{align*}
&D_t^{(\alpha,\beta)}(n,m)=\frac{(n+\alpha+\beta+2)(m+\alpha+\beta+1)}{2(n(n+\alpha+\beta+2)-m(m+\alpha+\beta+1))}\;\cdot\\
&\Bigg( \frac{t}{m+\alpha+\beta+1}I_t^{(\alpha+2,\beta+1,\alpha,\beta,\alpha+2,\beta+1)}{(n-1,m)}-\frac{t}{n+\alpha+\beta+2}I_t^{(\alpha+1,\beta,\alpha+1,\beta+1,\alpha+2,\beta+1)}{(n,m-1)}\\
&+\frac{1}{n+\alpha+\beta+2}I_t^{(\alpha+1,\beta,\alpha+1,\beta+1,\alpha+1,\beta+1)}{(n,m-1)}\Bigg), \quad t>0,
\end{align*}
where, as in \cite{ACL1},
$$
I_t^{(a,b,A,B,c,d)}{(k,l)}=\int_{-1}^1e^{-t(1-x)}P_{k}^{(a,b)}(x)P_{l}^{(A,B)}(x)(1-x)^{c}(1+x)^{d}dx,\quad k,l\in\N \text{ and } t>0.
$$
We have that
$$
n(n+\alpha+\beta+2)-m(m+\alpha+\beta+1)=(n-m)(n+m+\alpha+\beta+1)+n.
$$
Then, 
\begin{align*}
|n(n+\alpha+\beta+2)-m(m+\alpha+\beta+1)|&=\left\{\begin{array}{ll}
(n-m)(n+m+\alpha+\beta+1)+n, &\quad n>m\\
(m-n)(n+m+\alpha+\beta+1)-n, &\quad n<m
\end{array}\right.\\
&\ge \left\{\begin{array}{ll}
(n-m)(n+m+\alpha+\beta+1)+n, &\quad n>m\\
{(m-n)(m+\alpha+\beta+1)}, &\quad n<m
\end{array}\right.
\end{align*}
It follows that, for $k=n,m$,
$$
\frac{k+\alpha+\beta+2}{|n(n+\alpha+\beta+2)-m(m+\alpha+\beta+1)|}\le \frac{C}{|n-m|}.
$$
Then, 
\begin{align}\label{eq2.14}
D_t^{(\alpha,\beta)}&(n,m)=r_{n,m}^1 tI_t^{(\alpha+2,\beta+1,\alpha,\beta,\alpha+2,\beta+1)}{(n-1,m)}\nonumber\\&-r_{n,m}^2 tI_t^{(\alpha+1,\beta,\alpha+1,\beta+1,\alpha+2,\beta+1)}{(n,m-1)}
+r_{n,m}^3I_t^{(\alpha+1,\beta,\alpha+1,\beta+1,\alpha+1,\beta+1)}{(n,m-1)},
\end{align}
where $t>0$ and $|r_{n,m}^j|\le \frac{C}{|n-m|},$ $j=1,2,3.$

We have the following properties
\begin{enumerate}
	\item[(a)] Suppose that $n=m+k$, {$k\in\N$}. It follows that 
	\begin{align*}
	(n+\alpha+\beta+3)(n-1)-m&(m+\alpha+\beta+1)=(k{+\alpha+\beta+3})(m+k-1)\\&+(m+\alpha+\beta+3)(m+k-1)-m(m+\alpha+\beta+1)\\&\ge km,
	\end{align*}
	and
	\begin{align*}
	n(n+\alpha+\beta+2)-(m-1)(m+\alpha+\beta+2)&={(k+m)}({k+m}+\alpha+\beta+2)\\&-(m-1)(m+\alpha+\beta+2)\ge km,
	\end{align*}
	\item[(b)] Suppose that $m=n+k$, {$k\in\N$}. We get 
	\begin{align*}
	(n+\alpha+\beta+3)(n-1)-m(m+\alpha+\beta+1)&=(n+\alpha+\beta+3)(n-1)\\&\quad-(n+k)(n+k+\alpha+\beta+1)\\
	&=n-(\alpha+\beta+3)-k(2n+k+\alpha+\beta+1)\\&\le -kn,
	\end{align*}
	and
		\begin{align*}
	n(n+\alpha+\beta+2)-(m-1)(m+\alpha+\beta+2)&=n(n+\alpha+\beta+2)\\&\quad-(n+k-1)(n+k+\alpha+\beta+2)\\
	&=-nk-(k-1)(n{+}k+\alpha+\beta+2)\le -kn.
	\end{align*}
\end{enumerate}
By using again \cite[Lemma 5.1 (a)]{ACL1}, since $n\sim m$, $(a)$ y $(b)$ lead to  

\begin{align*}
|\partial_t&D_t^{(\alpha,\beta)}(n,m)|\le \frac{C}{|n-m|^2}\Big( t^2[|I_t^{(\alpha+3,\beta+2,\alpha,\beta,\alpha+4,\beta+2)}{(n-2,m)}|\\
&+|I_t^{(\alpha+2,\beta+1,\alpha+1,\beta+1,\alpha+4,\beta+2)}{(n-1,m-1)}|+|I_t^{(\alpha{+1},\beta,\alpha+2,\beta+2,\alpha+4,\beta+2)}{(n,m-2)}|]\\
&	t[|I_t^{(\alpha+3,\beta+2,\alpha,\beta,\alpha+3,\beta+2)}{(n-2,m)}|+|I_t^{(\alpha+2,\beta+1,\alpha+1,\beta+1,\alpha+3,\beta+2)}{(n-1,m-1)}|\\
&+|I_t^{(\alpha+1,\beta,\alpha+2,\beta+2,\alpha+3,\beta+2)}{(n,m-2)}|+|I_t^{(\alpha+2,\beta+1,\alpha+1,\beta+1,\alpha+4,\beta+1)}{(n-1,m-1)}|\\
&+|I_t^{(\alpha+2,\beta+1,\alpha+1,\beta+1,\alpha+3,\beta+1)}{(n-1,m-1)}|]\\
&+|I_t^{(\alpha+2,\beta+1,\alpha+1,\beta+1,\alpha+2,\beta+2)}{(n-1,m-1)}|+|I_t^{(\alpha+2,\beta+1,\alpha+1,\beta+1,\alpha+3,\beta+1)}{(n-1,m-1)}|\\
&+|I_t^{(\alpha+1,\beta,\alpha+2,\beta+2,\alpha+2,\beta+2)}{(n,m-2)}|+|I_t^{(\alpha+2,\beta+1,\alpha+1,\beta+1,\alpha+2,\beta+1)}{(n-1,m-1)}|\Big).
\end{align*}
By using \eqref{eq2.6} and \eqref{eq2.7} and by proceeding as in the first part of the proof we can see that
$$
\int_0^\infty |\partial_tD_t^{(\alpha,\beta)}(n,m)|dt \le\frac{C}{\sqrt{nm}|n-m|^2}.
$$
On the other hand, as in \eqref{eq2.14}, we obtain
\begin{align*}
D_t^{(\alpha,\beta)}&(1,2)=r_{1,2}^1 tI_t^{(\alpha+2,\beta+1,\alpha,\beta,\alpha+2,\beta+1)}{(0,2)}\nonumber\\&-r_{1,2}^2 tI_t^{(\alpha+1,\beta,\alpha+1,\beta+1,\alpha+2,\beta+1)}{(1,1)}
+r_{1,2}^3I_t^{(\alpha+1,\beta,\alpha+1,\beta+1,\alpha+1,\beta+1)}{(1,1)}, \quad t>0,
\end{align*}
where $|r_{1,2}^j|\le C$, $j=1,2,3.$ Then, by using \cite[Lemma 5.1 (a) y (b)]{ACL1} and proceeding as above, we conclude that 
$$\int_0^\infty|\partial_tD_t^{(\alpha,\beta)}(1,2)|dt\le C.$$
Thus \eqref{eq2.2} is proved.

According to \cite[Theorem 2.1]{BCFR}, we conclude that the operator $\V_\rho(\{ {W}_t^{(\alpha,\beta)}\}_{t>0})$ can be extended from $\ell^p({\N_0},w)\cap {\ell^2(\N_0)}$ to $\ell^p({\N_0},w)$ as a bounded operator 
\begin{itemize}
	\item[(i)] from $\ell^p({\N_0},w)$ into itself, for every $1<p<\infty$ and $w\in A^p({\N_0})$, 
		\item[(ii)] from $\ell^1({\N_0},w)$ into $\ell^{1,\infty}(\N,w)$, for every  $w\in A^1({\N_0})$.
\end{itemize}
\edproof

\subsection{Proof of Theorem \eqref{teo1.1} for jump operators}\textcolor{white}{}\newline

According to \cite[p. {6712}]{JSW}, we have that
$$
{\lambda}(\Lambda(\{W_t^{(\alpha,\beta)}\}_{t>0},\lambda)(f))^{1/\rho}\le 2^{1+\frac{1}{\rho}}\V_\rho(\{ {W}_t^{(\alpha,\beta)}\}_{t>0})(f),\quad \lambda>0.
$$
Therefore, properties for $\lambda$-jump operators stated in Theorem \eqref{teo1.1} are consequences of the corresponding ones for the variation operators.
\edproof

\subsection{Proof of Theorem \eqref{teo1.1} for oscillation operators}\textcolor{white}{}\newline

By keeping the notation from subsection \ref{subsec2.1}, for every $n\in{\N_0},$ we have that
$$
\OO(\{\tilde{W}_t^{(\alpha,\beta)}\}_{t>0},\{t_j\}_{j\in\N})(f)(n)=\frac{1}{p_n^{(\alpha,\beta)}(1)}(\OO(\{W_t^{(\alpha,\beta)}\}_{t>0},\{t_j\}_{j\in\N})(p_{\cdot}^{(\alpha,\beta)}(1)f)(n),
$$
According to \cite[p. 20]{LX} (see also \cite[{Theorem 3.3}]{JRe}), the oscillation operator $\OO(\{\tilde{W}_t^{(\alpha,\beta)}\}_{t>0},\{t_j\}_{j\in\N})$ is bounded from $\ell^2({\N_0}, v^{(\alpha,\beta)})$ into itself. Then, the operator $\OO(\{W_t^{(\alpha,\beta)}\}_{t>0},\{t_j\}_{j\in\N})$ is bounded from  ${\ell^2(\N_0)}$ into itself.

Suppose that $g$ is a complex-valued function defined in $(0,\infty)$. We defime
$$
\|g\|_{\OO(\{t_j\}_{j\in\N})}=\left(\sum_{{j=1}}^{\infty}\sup_{t_{j+1}\le \epsilon_{j+1}<\epsilon_j\le t_{j}}|g({\epsilon_j})-g(\epsilon_{j+1})|^2 \right)^{1/2}.
$$

By identifying each pair of functions $g_1$ and $g_2$ such that $g_1-g_2$ is a constant, $\|\cdot\|_{\OO(\{t_j\}_{j\in\N})}$ is a norm in th space $F_{\OO(\{t_j\}_{j\in\N})}      $ of all complex functions $g$ defined on $(0,\infty)$ such that $ \|g\|_{\OO(\{t_j\}_{j\in\N})}<\infty.$

Thus, $(F_{\OO(\{t_j\}_{j\in\N})}  ,\|\cdot\|_{\OO(\{t_j\}_{j\in\N})})$ is a Banach space.

If $g$ is a complex function which is differentiable in $(0,\infty)$, we have that
\begin{align*}
\|g\|_{\OO(\{t_j\}_{j\in\N})}&=\left(\sum_{{j=1}}^{\infty}\sup_{t_{j+1}\le \epsilon_{j+1}<\epsilon_j\le t_{j}}\left|\int^{\epsilon_j}_{\epsilon_{j+1}}g'(s)ds\right|^2 \right)^{1/2}\\
&\le \left(\sum_{{j=1}}^{\infty}\sup_{t_{j+1}\le \epsilon_{j+1}<\epsilon_j\le t_{j}}\left(\int^{\epsilon_j}_{\epsilon_{j+1}}\left|g'(s)\right|ds\right)^2 \right)^{1/2}\\
&\le \int_0^\infty \left|g'(s)\right|ds.
\end{align*}
From the established  estimates in subsection \ref{subsec2.1}, we deduce that
\begin{equation*}
\| {K}_t^{(\alpha,\beta)}(n,m)\|_{\OO(\{t_j\}_{j\in\N})}\le \frac{C}{|n-m|}, \quad n,m\in{\N_0}, \: n\neq m,
\end{equation*}
and
\begin{equation*}
\|{K}_t^{(\alpha,\beta)}(n,m)-{K}_t^{(\alpha,\beta)}(l,m) \|_{\OO(\{t_j\}_{j\in\N})}\le C\frac{|n-l|}{|n-m|^2}, \quad |n-m|>2|n-l|,\: \frac{m}{2}\le n,l\le \frac{3m}{2}.
\end{equation*}

By using \cite[Theorem 1.1]{BCFR}, we conclude that the oscillation operator \newline $\OO(\{\tilde{W}_t^{(\alpha,\beta)}\}_{t>0},\{t_j\}_{j\in\N})$  
can be extended from $\ell^p({\N_0},w)\cap {\ell^2(\N_0)}$ to $\ell^p({\N_0},w)$ as a bounded operator 
\begin{itemize}
	\item[(i)] from $\ell^p({\N_0},w)$ into itself, for every $1<p<\infty$ and $w\in A^p({\N_0})$, 
	\item[(ii)] from $\ell^1({\N_0},w)$ into $\ell^{1,\infty}({\N_0},w)$, for every  $w\in A^1({\N_0})$.
\end{itemize}
\edproof

\section{Proof of Theorem \ref{teo1.2}}

\subsection{The operators $S_{\{a_j\}_{j\in{\Z}},N}^{\{b_j\}_{j\in{\Z}}}(\{W_t^{(\alpha,\beta)}\}_{t>0})$}
\textcolor{white}{}\newline

In this section we shall prove the following result.
\begin{theorem}\label{teo3.1}
	Let $\alpha, \beta\ge -1/2$. Assume that $\{a_j\}_{j\in\Z}$ is a $\rho$-lacunary sequence in $(0,\infty)$ with $\rho>1$ and $\{b_j\}_{j\in\Z}$ is a bounded sequence of real numbers. For every $N=(N_1,N_2)\in\Z^2$, $N_1<N_2$, the operator $S_{\{a_j\}_{j\in{\Z}},N}^{\{b_j\}_{j\in{\Z}}}(\{W_t^{(\alpha,\beta)}\}_{t>0})$ is bounded from $\ell^p({\N_0},w)$ into itself, for every $1<p<\infty$ and $w\in A^p({\N_0})$,  and from $\ell^1({\N_0},w)$ into $\ell^{1,\infty}({\N_0},w)$, for every  $w\in A^1({\N_0})$. Furthermore, for every $1<p<\infty$ and $w\in A^p({\N_0})$, 
	$$
	\sup_{\substack{N=(N_1,N_2)\in\Z^2\\N_1<N_2}}\left\| S_{\{a_j\}_{j\in{\Z}},N}^{\{b_j\}_{j\in{\Z}}}(\{W_t^{(\alpha,\beta)}\}_{t>0})\right\|_{\ell^p({\N_0},w)\to \ell^p({\N_0},w)}<\infty,
	$$ and, for every $w\in A^1({\N_0})$, 
	$$
	\sup_{\substack{N=(N_1,N_2)\in\Z^2\\N_1<N_2}}\left\| S_{\{a_j\}_{j\in{\Z}},N}^{\{b_j\}_{j\in{\Z}}}(\{W_t^{(\alpha,\beta)}\}_{t>0})\right\|_{\ell^1({\N_0},w)\to \ell^{1,\infty}({\N_0},w)}<\infty.
	$$
\end{theorem}
\begin{proof}
	Let $N=(N_1,N_2)\in\Z^2$ with $N_1<N_2$. By proceeding as in the proof of \cite[Theorem 2.1, p. 627]{RZ} and by using the $(\alpha,\beta)$-Fourier transform we can see that
	$$
	\left\| S_{\{a_j\}_{j\in{\Z}},N}^{\{b_j\}_{j\in{\Z}}}(\{W_t^{(\alpha,\beta)}\}_{t>0})(f)\right\|_{{\ell^2(\N_0)}}\le C\|f\|_{{\ell^2(\N_0)}},\quad f\in{\ell^2(\N_0)},
	$$
	where $C>0$ does not depend on $N$.

	We have that, for every $f\in {\ell^2(\N_0)}$,
	$$
	S_{\{a_j\}_{j\in{\Z}},N}^{\{b_j\}_{j\in{\Z}}}(\{W_t^{(\alpha,\beta)}\}_{t>0})(f)(n)=\sum_{m\in{\N_0}}f(m)\mathcal{Q}_N^{(\alpha,\beta)}(n,m), \quad n\in{\N_0},
	$$
	where 
	$$
	\mathcal{Q}_N^{(\alpha,\beta)}(n,m)=\sum_{j=N_1}^{N_2}b_j(K_{a_{j+1}}^{(\alpha,\beta)}(n,m)-K_{a_{j}}^{(\alpha,\beta)}(n,m)), \quad n,m\in{\N_0}.
	$$
	According to \eqref{eq2.3}, we obtain
	$$
	|	\mathcal{Q}_N^{(\alpha,\beta)}(n,m)|\le \|b_j\|_{\ell^\infty(\Z)}\int_0^\infty|\partial_t K_t ^{(\alpha,\beta)}(n,m)|dt,\quad n,m\in{\N_0}.
	$$
	In the proof of \eqref{eq2.1} we established that
	$$
	\int_0^\infty|\partial_t K_t^{(\alpha,\beta)}(n,m)|dt\le\frac{C}{|n-m|}, \quad n,m\in{\N_0}, \: n\neq m.
	$$
	Then 
	\begin{equation}\label{eq3.A}
	|	\mathcal{Q}_N^{(\alpha,\beta)}(n,m)|\le\frac{C}{|n-m|}, \quad n,m\in{\N_0}, \: n\neq m,
	\end{equation}
	where $C>0$ does not depend on $N$.
	
	Also, by proceeding as in the proof of \eqref{eq2.2}, we can see that
	\begin{align}\label{eq3.B} |\mathcal{Q}_N^{(\alpha,\beta)}(n,m)-	\mathcal{Q}_N^{(\alpha,\beta)}(l,m) |\le C\frac{|n-l|}{|n-m|^2}, \quad |n-m|>2|n-l|,\: \frac{m}{2}\le n,l\le \frac{3m}{2},
\end{align}
being $C$ independent of $N$.

The proof can be finished by using \cite[Theorem 2.1]{BCFR}.
	\end{proof}
\vspace{0.4cm}

For every $N=(N_1,N_2)\in\Z^2$ with $N_1<N_2$, we define
$$
	S_{\{a_j\}_{j\in{\Z}},N,loc}^{\{b_j\}_{j\in{\Z}}}(\{W_t^{(\alpha,\beta)}\}_{t>0})(f)(n)= S_{\{a_j\}_{j\in{\Z}},N}^{\{b_j\}_{j\in{\Z}}}(\{W_t^{(\alpha,\beta)}\}_{t>0})(f\chi_{{\frac{n}{2}\le m\le \frac{3n}{2}}})(n),\quad n\in{\N_0},
$$
 and
 \begin{align*}
 S_{\{a_j\}_{j\in{\Z}},N,glob}^{\{b_j\}_{j\in{\Z}}}(\{W_t^{(\alpha,\beta)}\}_{t>0})(f)&= S_{\{a_j\}_{j\in{\Z}},N}^{\{b_j\}_{j\in{\Z}}}(\{W_t^{(\alpha,\beta)}\}_{t>0})(f)\\&\quad-
 S_{\{a_j\}_{j\in{\Z}},N,loc}^{\{b_j\}_{j\in{\Z}}}(\{W_t^{(\alpha,\beta)}\}_{t>0})(f).
 \end{align*}
\begin{corollary}
	Properties in Theorem \ref{teo3.1} hold for $S_{\{a_j\}_{j\in{\Z}},N,loc}^{\{b_j\}_{j\in{\Z}}}(\{W_t^{(\alpha,\beta)}\}_{t>0})$ and $S_{\{a_j\}_{j\in{\Z}},N,glob}^{\{b_j\}_{j\in{\Z}}}(\{W_t^{(\alpha,\beta)}\}_{t>0})$.
\end{corollary}

\begin{proof}
	Let $N\in\Z$. According to \eqref{eq3.A}, we have that 
	$$
	|S_{\{a_j\}_{j\in{\Z}},N,glob}^{\{b_j\}_{j\in{\Z}}}(\{W_t^{(\alpha,\beta)}\}_{t>0})(f)(n)|\le C\left(\frac{1}{n}\sum_{m=0}^{n-1}|f(m)|+\sum_{m=n+1}^\infty \frac{|f(m)|}{m}\right),\quad n\in{\N_0},
	$$
	where $C>0$ does not depend on $N$. The first term in the right hand side does not appear when $n=0$. By using $\ell^p$-boundedness properties of discrete Hardy operators we can deduce that the corresponding properties for $S_{\{a_j\}_{j\in{\Z}},N,glob}^{\{b_j\}_{j\in{\Z}}}(\{W_t^{(\alpha,\beta)}\}_{t>0})$. The proof can be finished by using Theorem \ref{teo3.1}.
\end{proof}

\subsection{Some auxiliary results}

In order to prove a Cotlar inequality for $T_*^{(\alpha,\beta)}$, we need the following results.
\begin{proposition}\label{prop3.1}
	Let  $\alpha, \beta\ge -1/2$. Then,
	$$
	\sup_{t>0}|\partial_tK_t^{(\alpha,\beta)}(n,m)|\le\frac{C}{|n-m|^3},\quad n,m\in{\N_0},\quad n\neq m.
	$$
\end{proposition}
\begin{proof}
	We will use \cite[Lemma 5.1]{ACL1} several times. Let $n,m\in\N$, $n,m\ge 3,$ $n\neq m$. According to \cite[Lemma 5.1 (a)]{ACL1}, we get
	\begin{enumerate}
		\item[(i)]\begin{small}\begin{align*}I_t^{(\alpha,\beta,\alpha,\beta,\alpha,\beta)}&(n,m)=\frac{(n+\alpha+\beta+1)(m+\alpha+\beta+1)}{2(n-m)(n+m+\alpha+\beta+1)}\;t\Bigg( \frac{1}{m+\alpha+\beta+1}\cdot\\
			&I_t^{(\alpha+1,\beta+1,\alpha,\beta,\alpha+1,\beta+1)}{(n-1,m)}-\frac{1}{n+\alpha+\beta+1}I_t^{(\alpha,\beta,\alpha+1,\beta+1,\alpha+1,\beta+1)}{(n,m-1)}\Bigg).\end{align*}\end{small}
			\item[(ii)]\begin{small}\begin{align*}I_t^{(\alpha+1,\beta+1,\alpha,\beta,\alpha+1,\beta+1)}&(n-1,m)=\frac{(n+\alpha+\beta+2)(m+\alpha+\beta+1)}{2((n-m)(n+m+\alpha+\beta+1)-(\alpha+\beta+2))}\\
				&\Bigg( \frac{t}{m+\alpha+\beta+1}
			I_t^{(\alpha+2,\beta+2,\alpha,\beta,\alpha+2,\beta+2)}{(n-2,m)}\\&-\frac{t}{n+\alpha+\beta+2}I_t^{(\alpha+1,\beta+1,\alpha+1,\beta+1,\alpha+2,\beta+2)}{(n-1,m-1)}\\
			&	+\frac{1}{{n}+\alpha+\beta+2}
			I_t^{(\alpha+1,\beta+1,\alpha+1,\beta+1,\alpha+1,\beta+2)}{(n-1,m-1)}\\
			&{-}\frac{1}{n+\alpha+\beta+2}I_t^{(\alpha+1,\beta+1,\alpha+1,\beta+1,\alpha+2,\beta+1)}{(n-1,m-1)}\Bigg).\end{align*}\end{small}
			\item[(iii)]\begin{small}\begin{align*}I_t^{(\alpha,\beta,\alpha+1,\beta+1,\alpha+1,\beta+1)}&(n,m-1)=\frac{(n+\alpha+\beta+1)(m+\alpha+\beta+{2})}{2((n-m)(n+m+\alpha+\beta+1)+(\alpha+\beta+2))}\\
			&\Bigg( \frac{t}{m+\alpha+\beta+2}
			I_t^{(\alpha+1,\beta+1,\alpha+1,\beta+1,\alpha+2,\beta+2)}{(n-1,m-1)}\\&-\frac{1}{m+\alpha+\beta+2}I_t^{(\alpha+1,\beta+1,\alpha+1,\beta+1,\alpha+1,\beta+2)}{(n-1,m-1)}\\
			&	+\frac{1}{{m}+\alpha+\beta+2}
			I_t^{(\alpha+1,\beta+1,\alpha+1,\beta+1,\alpha+2,\beta+1)}{(n-1,m-1)}\\&-\frac{t}{n+\alpha+\beta+1}I_t^{(\alpha,\beta,\alpha+2,\beta+2,\alpha+2,\beta+2)}{(n,m-2)}\Bigg).\end{align*}\end{small}
		\end{enumerate}
	We apply again \cite[Lemma 5.1 (a)]{ACL1} to each of the four terms in the right hand side in (ii) and (iii). We obtain that
	\begin{align*}
	I_t^{(\alpha,\beta,\alpha,\beta,\alpha,\beta)}&(n,m)=t^3\sum_{j\in J_1}c_{j1}(n,m)I_t^{(a_{j1},b_{j1},A_{j1},B_{j1},\eta_{j1},\gamma_{j1})}(l_{j1},k_{j1})\\
	&+t^2\sum_{j\in J_2}c_{j2}(n,m)I_t^{(a_{j2},b_{j2},A_{j2},B_{j2},\eta_{j2},\gamma_{j2})}(l_{j2},k_{j2})\\
		&+t\sum_{j\in J_3}c_{j3}(n,m)I_t^{(a_{j3},b_{j3},A_{j3},B_{j3},\eta_{j3},\gamma_{j3})}(l_{j3},k_{j3}),\quad t>0.
	\end{align*}
	Here, $J_1=J_3=\{n\in\N:\: 1\le n\le 8\}$ and $J_2=\{n\in\N:\: 1\le n\le 20\}$, being
	\begin{itemize}
		\item $|c_{ji}(n,m)|\le \frac{C}{|n-m|^3}$, $j\in J_i$, $i=1,2,3$.
		\item $(l_{ji},k_{ji})\in\{(l,k):\: l,k\in{\N_0},\: n-3\le l\le n,\: m-3\le k\le m\}$, $j\in J_i$, $i=1,2,3.$
		\item $\eta_{j1}=\alpha+3$, $\gamma_{j1}=\beta+3$, $j\in J_1$.
		\item $\eta_{j3}=\alpha+2$, $\gamma_{j3}=\beta+2$, $j\in J_3$.
		\item $(\eta_{j2},\gamma_{j2})\in\{(\alpha+2,\beta+3),(\alpha+3,\beta+2)\}$, $j\in J_2$.
		\item $a_{ji}+A_{ji}=2\alpha+3$, $b_{ji}+B_{ji}=2\beta+3$, $j\in J_i$, $i=1,2,3$.
	\end{itemize}
According to \eqref{eq2.4}, we obtain
\begin{align*}
|\partial_t	&I_t^{(\alpha,\beta,\alpha,\beta,\alpha,\beta)}(n,m)|\le \frac{C}{|n-m|^3}\Bigg(t^2\int_{-1}^1e^{-t(1-x)}(1-x)(1+x)dx\\
&\quad+t\int_{-1}^1e^{-t(1-x)}(1+x)dx+t\int_{-1}^1e^{-t(1-x)}(1-x)dx+\int_{-1}^1e^{-t(1-x)}dx\Bigg)\\
&\le \frac{C}{|n-m|^3}\Bigg(\int_0^{2t}e^{-u}u\;du+\int_0^{2t}e^{-u}du+\frac{1}{t}\int_0^{2t}e^{-u}u\;du+\frac{1}{t}\int_0^{2t}e^{-u}\;du\Bigg)\\
&\le\frac{C}{|n-m|^3},\quad t>0.
\end{align*}
When $n,m\in{\N_0}$, $n<3$ or $m<3$, we can proceed in a similar way by using \cite[Lemma 5.1 (a),(b) and (c)]{ACL1}.

\end{proof}

We say that a positive sequence is $(\lambda, \lambda^2)$-lacunary  with $\lambda>1$ when $\lambda\le \frac{a_{j+1}}{a_j}\le \lambda^2$, $j\in\Z$.

\begin{proposition}\label{prop3.2}
	Suppose that $\{a_j\}_{j\in\Z}$ is a $(\lambda, \lambda^2)$-lacunary sequence and $\{v_j\}_{j\in \Z}$ is a bounded complex sequence. Then,
	\begin{itemize}
		\item[(i)] $\displaystyle\left|\sum_{j=k}^{M}v_j(K_{a_{j+1}}^{(\alpha,\beta)}(n,m)-K_{a_{j}}^{(\alpha,\beta)}(n,m)) \right|\le \frac{C}{\sqrt{{a_k}}}, \quad k,M\in\Z, \; k<M\;\: n,m\in{\N_0},$
		\item[(ii)]$\displaystyle\left|\sum_{j=-M}^{l-1}v_j(K_{a_{j+1}}^{(\alpha,\beta)}(n,m)-K_{a_{j}}^{(\alpha,\beta)}(n,m)) \right|\le \frac{C}{\sqrt{{a_k}}}\lambda^{-(k-{l}+1)},$ when  $k,M,l\in\Z,\;$ \newline  $k>l>-M$, $C>0$ and $n,m\in{\N_0},$ $|n-m|\ge C\sqrt{a_k}$.
	\end{itemize}
\end{proposition}

\begin{proof}
	$(i)$ Let $j\in\Z$. By using the mean value theorem, we obtain
	$$
	K_{a_{j+1}}^{(\alpha,\beta)}(n,m)-K_{a_{j}}^{(\alpha,\beta)}(n,m)=(a_{j+1}-a_{j})\partial_t 	K_{t}^{(\alpha,\beta)}(n,m)|_{t=c_j},
	$$
	for a certain $c_j\in (a_j, a_{j+1})$. According to \eqref{eq2}, since $w_k^{(\alpha,\beta)}\sim\sqrt{k+1}$, $k\in{\N_0}$, we get
	\begin{align*}
	|\partial_t 	K_{t}^{(\alpha,\beta)}(n,m)|&\le C \int_{-1}^1 e^{-t(1-x)}\sqrt{\frac{1-x}{1+x}}dx\\&\le C \left(e^{-t}+\int_0^1 e^{-tz}\sqrt{z}dz \right)\le C (e^{-t}+t^{-3/2})\le\frac{C}{t^{3/2}}, \quad n,m\in{\N_0}  \text{ and } t>0.
	\end{align*}
	Then,
	\begin{align*}
	|K_{a_{j+1}}^{(\alpha,\beta)}(n,m)-K_{a_{j}}^{(\alpha,\beta)}(n,m)|&\le C \frac{|a_{j+1}-a_{j}|}{a_{j}^{3/2}}\\
	&\le C\frac{\lambda^2-1}{\sqrt{a_j}}, \quad n,m\in{\N_0}.
	\end{align*}
	It follows that, for every $k,M\in\Z$, $k<M$, $n,m\in{\N_0}$,
	\begin{align*}\left|\sum_{j=k}^{M}v_j(K_{a_{j+1}}^{(\alpha,\beta)}(n,m)-K_{a_{j}}^{(\alpha,\beta)}(n,m)) \right|&\le C\sum_{j=k}^M \frac{1}{\sqrt{{a_k}}}\le \frac{C}{\sqrt{{a_k}}}\sum_{j=k}^M \sqrt{\frac{a_k}{{{a_j}}}}\\&\le \frac{C}{\sqrt{{a_k}}}.	\end{align*}
		$(ii)$ Let $j\in\Z$. By using Proposition \ref{prop3.1} and again the mean value theorem, we obtain
			\begin{align*}
		|K_{a_{j+1}}^{(\alpha,\beta)}(n,m)-K_{a_{j}}^{(\alpha,\beta)}(n,m)|&\le C \frac{|a_{j+1}-a_{j}|}{|n-m|^{3}}\le C\frac{a_{j}}{|n-m|^{3}}, \quad n,m\in{\N_0}.
		\end{align*}
		Then, 
		\begin{align*}\left|\sum_{j=-M}^{l-1}v_j(K_{a_{j+1}}^{(\alpha,\beta)}(n,m)-K_{a_{j}}^{(\alpha,\beta)}(n,m)) \right|&\le C\sum_{j=-M}^{l-1} \frac{a_{j}}{|n-m|^{3}}\le C\sum_{j=-M}^{l-1} \frac{a_{j}}{a_{k}^{3/2}}\\&\le \frac{C}{\sqrt{{a_k}}}\lambda^{-(k-{l}+1)},	\end{align*}
		provided that $k,M,l\in\Z$, $k\ge l>-M$, $n,m\in{\N_0}$, $|n-m|>C\sqrt{a_k}$, with $C>0$.
	\end{proof}
\vspace{0.3cm}

		By $\mathcal{M}$ we denote the centered Hardy-Littlewood maximal function, given by
		$$
		\mathcal{M}(f)(n)=\sup_{r>0}\frac{1}{\mu_d(B_{{\N_0}}(n,r))}\sum_{m\in B_{{\N_0}}(n,r)}|f(m)|,\quad n\in{\N_0}.
		$$
		Here, $B_{{\N_0}}(n,r)=\{m\in{\N_0}:\: |m-n|<r\}$, $n\in{\N_0}$ and $r>0$. For every $1<q<\infty$ we consider $\mathcal{M}_q$, defined by
		$$
		\mathcal{M}_q(f)=\left( 	\mathcal{M}(|f|^q)\right)^{1/q}.
		$$
		We now prove a Cotlar type inequality for the local maximal operator
	
		\begin{align*}
		S_{\{a_j\}_{j\in\Z},*,M,loc}^{\{b_j\}_{j\in\Z}}&(\{W_t^{(\alpha,\beta)}\}_{t>0})(f)(n)\\&=\sup_{\substack{N=(N_1,N_2)\\
				N_1,N_2\in\Z,\: -M\le N_1<N_2\le M}}|S_{\{a_j\}_{j\in\Z},N}^{\{b_j\}_{j\in{\Z}}}(\{W_t^{(\alpha,\beta)}\}_{t>0})(f\chi_{{\frac{n}{2}\le m\le \frac{3n}{2}}})(n)|
					\end{align*}
		for every $M\in\N$.
	
		\begin{proposition}\label{prop3.3}
			Suppose that $\{a_j\}_{j\in\Z}$ is a $(\lambda, \lambda^2)$-lacunary sequence $\{v_j\}_{j\in \Z}$ is a bounded complex sequence  and $1<q<\infty$. Then, there exists $C>0$ such that, for every $M\in\N$, 
			\begin{align*}
				S_{\{a_j\}_{j\in\Z},*,M,loc}^{\{b_j\}_{j\in\Z}}(\{W_t^{(\alpha,\beta)}\}_{t>0})(f)(n)&\le C\Bigg( 	\mathcal{M}(	S_{\{a_j\}_{j\in\Z},(-M,M),loc}^{\{b_j\}_{j\in\Z}}(\{W_t^{(\alpha,\beta)}\}_{t>0})(f))\\&\qquad+	\mathcal{M}_q(f)\Bigg).
				\end{align*}
		\end{proposition}
\begin{proof}
	In order to prove this property we can proceed adapting to our context the proof of \cite[Theorem 3.11]{TZ2}. The properties that we need have been established in Proposition \ref{prop3.2}, {\eqref{eq3.A},\eqref{eq3.B}} and {Theorem \ref{teo3.1}.} We now sketch the proof.
	
	Let $M\in\N$. For every $N=(N_1,N_2)$ with $-M<N_1<N_2<M$, we can write

	\begin{align*}
	S_{\{a_j\}_{j\in{\Z}},N}^{\{b_j\}_{j\in{\Z}}}&(\{W_t^{(\alpha,\beta)}\}_{t>0})(f\chi_{\frac{n}{2}\le m\le \frac{3n}{2}})(n)\\
	&=	S_{\{a_j\}_{j\in{\Z}},(N_1,M)}^{\{b_j\}_{j\in{\Z}}}(\{W_t^{(\alpha,\beta)}\}_{t>0})(f\chi_{\frac{n}{2}\le m\le \frac{3n}{2}})(n)\\&\quad-S_{\{a_j\}_{j\in{\Z}},(N_2+1,M)}^{\{b_j\}_{j\in{\Z}}}(\{W_t^{(\alpha,\beta)}\}_{t>0})(f\chi_{\frac{n}{2}\le m\le \frac{3n}{2}})(n), \quad n\in{\N_0}.
	\end{align*}
	We are going to see that  there exists $C>0$ such that 
	\begin{align*}
	|&S_{\{a_j\}_{j\in{\Z}},(l,M)}^{\{b_j\}_{j\in{\Z}}}(\{W_t^{(\alpha,\beta)}\}_{t>0})(f\chi_{\frac{n}{2}\le m\le \frac{3n}{2}})(n)|\\
	&\le  C\left( 	\mathcal{M}(	S_{\{a_j\}_{j\in{\Z}},(-M,M)}^{\{b_j\}_{j\in{\Z}}}(\{W_t^{(\alpha,\beta)}\}_{t>0})(f))(n)+	\mathcal{M}_q(f)(n)\right),
	\end{align*}
	for every $l\in\Z$, $-M<l<M$ and $n\in\N$. Here, $C$ does not depend on $n\in{\N_0}$, $M\in\N$ and $l\in\Z$, $-M<l<M$.
	
	Assume that $n\in{\N_0}$ and $l\in\Z$, $-M<l<M$. We decompose $f$ as follows
	$$
	f=f\chi_{B_{\N_0} (n,\sqrt{a_l})}+f\chi_{B_{\N_0} (n,\sqrt{a_l})^c}=:f_1+f_2.
	$$
	We have that
	\begin{align*}
	|&S_{\{a_j\}_{j\in\Z},(l,M)}^{\{b_j\}_{j\in\Z}}(\{W_t^{(\alpha,\beta)}\}_{t>0})(f\chi_{\frac{n}{2}\le m\le \frac{3n}{2}})(n)|\\
	&\le	|S_{\{a_j\}_{j\in\Z},(l,M)}^{\{b_j\}_{j\in\Z}}(\{W_t^{(\alpha,\beta)}\}_{t>0})(f_1\chi_{\frac{n}{2}\le m\le \frac{3n}{2}})(n)|\\&\quad+	|S_{\{a_j\}_{j\in\Z},(l,M)}^{\{b_j\}_{j\in\Z}}(\{W_t^{(\alpha,\beta)}\}_{t>0})(f_2\chi_{\frac{n}{2}\le m\le \frac{3n}{2}})(n)|=:A(l,M,n)+B(l,M,n).
		\end{align*}
		According to Proposition \ref{prop3.2} (i), we obtain
		$$
		A(l,M,n)\le \frac{C}{\sqrt{a_l}}\sum_{k\in B_l}|f_1(k)|\le C\mathcal{M}(f)(n).
		$$
		On the other hand, we can write
			\begin{align*}
			B(l,M,n)\le \frac{C}{\sqrt{a_{l-1}}}\sum_{|k-n|\le \frac{1}{2}\sqrt{a_{{l-1}}}}\Bigg(&|	S_{\{a_j\}_{j\in\Z},(-M,M)}^{\{b_j\}_{j\in\Z}}(\{W_t^{(\alpha,\beta)}\}_{t>0})(f\chi_{{\frac{k}{2}\le m\le \frac{3k}{2}}})(k)|\\
			&+|	S_{\{a_j\}_{j\in\Z},(-M,M)}^{\{b_j\}_{j\in\Z}}(\{W_t^{(\alpha,\beta)}\}_{t>0})(f_1\chi_{{\frac{k}{2}\le m\le \frac{3k}{2}}})(k)|\\
			&+|	S_{\{a_j\}_{j\in\Z},(l,M)}^{\{b_j\}_{j\in\Z}}(\{W_t^{(\alpha,\beta)}\}_{t>0})(f_2\chi_{{\frac{k}{2}\le m\le \frac{3k}{2}}})(k)\\&\qquad-S_{\{a_j\}_{j\in\Z},(l,M)}^{\{b_j\}_{j\in\Z}}(\{W_t^{(\alpha,\beta)}\}_{t>0})(f_2\chi_{{\frac{n}{2}\le m\le \frac{3n}{2}}})(n)|\\
			&+|	S_{\{a_j\}_{j\in\Z},(-M,l-1)}^{\{b_j\}_{j\in\Z}}(\{W_t^{(\alpha,\beta)}\}_{t>0})(f_2\chi_{{\frac{k}{2}\le m\le \frac{3k}{2}}})(k)|
			\Bigg)\\
			&=:\sum_{i=1}^4 B_i(l,M,n),
				\end{align*}
				with the obvious understanding for the four sums when $l=-M$.
				
				We now estimate $B_i(l,M,n)$, $i=1,2,3,4$.
				\begin{itemize}
					\item[(i)] It is clear that
					$$
					B_1(l,M,n)\le 	C\mathcal{M}(	S_{\{a_j\}_{j\in\Z},(-M,M),loc}^{\{b_j\}_{j\in\Z}}(\{W_t^{(\alpha,\beta)}\}_{t>0})(f))(n).
					$$
				\item[(ii)]	Since the family $\left\{S_{\{a_j\}_{j\in\Z},N}^{\{b_j\}_{j\in\Z}}(\{W_t^{(\alpha,\beta)}\}_{t>0})\right\}_{\substack{N=(N_1,N_2)\in\Z^2\\N_1<N_2}}$ of operators is uniformly bounded from $L^q({\N_0})$ into itself, $\left\{S_{\{a_j\}_{j\in\Z},N,loc}^{\{b_j\}_{j\in\Z}}(\{W_t^{(\alpha,\beta)}\}_{t>0})\right\}_{\substack{N=(N_1,N_2)\in\Z^2\\N_1<N_2}}$ is also uniformly bounded from $L^q({\N_0})$ into itself. Then, by using H\"older inequality and by taking into account that is a $(\lambda, \lambda^2)$-lacunary sequence, we obtain that
				$$
					B_2(l,M,n)\le C	\mathcal{M}_q(f)(n).
				$$
				\item[(iii)] By using {\eqref{eq3.A} and \eqref{eq3.B}}, we can prove, by proceeding as in the proof of \cite[{(18)}]{BCFR} that
				$$
				B_3(l,M,n)\le C	\mathcal{M}(f)(n).
				$$
					\item[(iv)] By Proposition \ref{prop3.2} (ii), we deduce that
					$$
							B_4(l,M,n)\le C	\mathcal{M}(f)(n).
					$$
					\end{itemize}
					By combining (i)-(iv), it follows that
				\begin{align*}
				B(l,M,n)
			\le  C\left( 	\mathcal{M}(	S_{\{a_j\}_{j\in\Z},(-M,M),loc}^{\{b_j\}_{j\in\Z}}(\{W_t^{(\alpha,\beta)}\}_{t>0})(f))(n)+	\mathcal{M}_q(f)(n)\right).
			\end{align*}
			Thus, we conclude that
					\begin{align*}
				|S_{\{a_j\}_{j\in{\Z}},(l,M),loc}^{\{b_j\}_{j\in{\Z}}}&(\{W_t^{(\alpha,\beta)}\}_{t>0})(f)(n)|\\&\le  C\left( 	\mathcal{M}(	S_{\{a_j\}_{j\in\Z},(-M,M),loc}^{\{b_j\}_{j\in\Z}}(\{W_t^{(\alpha,\beta)}\}_{t>0})(f))(n)+	\mathcal{M}_q(f)(n)\right).
				\end{align*}
\end{proof}

\subsection{Proof of Theorem \ref{teo1.2}}

Let $M\in\N$. For every $n\in{\N_0}$, we can write
\begin{align*}
S_{\{a_j\}_{j\in{\Z}},*,M}^{\{b_j\}_{j\in{\Z}}}(\{W_t^{(\alpha,\beta)}\}_{t>0})(f)(n)&\le  S_{\{a_j\}_{j\in{\Z}},*,M,loc}^{\{b_j\}_{j\in{\Z}}}(\{W_t^{(\alpha,\beta)}\}_{t>0})(f)(n)\\&\quad+S_{\{a_j\}_{j\in{\Z}},*,M,glob}^{\{b_j\}_{j\in{\Z}}}(\{W_t^{(\alpha,\beta)}\}_{t>0})(f)(n),
\end{align*}
where 
$$
S_{\{a_j\}_{j\in{\Z}},*,M,loc}^{\{b_j\}_{j\in{\Z}}}(\{W_t^{(\alpha,\beta)}\}_{t>0})(f)(n)=\sup_{\substack{N=(N_1,N_2)\\
		N_1,N_2\in\Z,\\ -M\le N_1<N_2\le M}}|S_{\{a_j\}_{j\in{\Z}},N}^{\{b_j\}_{j\in{\Z}}}(\{W_t^{(\alpha,\beta)}\}_{t>0})(f\chi_{[\frac{{n}}{2}, \frac{{3n}}{2}]})(n)|,
$$
and
$$
S_{\{a_j\}_{j\in{\Z}},*,M,glob}^{\{b_j\}_{j\in{\Z}}}(\{W_t^{(\alpha,\beta)}\}_{t>0})(f)(n)=\sup_{\substack{N=(N_1,N_2)\\
		N_1,N_2\in\Z,\\ -M\le N_1<N_2\le M}}|S_{\{a_j\}_{j\in{\Z},N}}^{\{b_j\}_{j\in{\Z}}}(\{W_t^{(\alpha,\beta)}\}_{t>0})(f(1-\chi_{[\frac{{n}}{2}, \frac{{3n}}{2}]}))(n)|.
$$
According to {\eqref{eq3.A}}, there exists $C>0$ such that
\begin{align*}
|S_{\{a_j\}_{j\in{\Z}},*,M,glob}^{\{b_j\}_{j\in{\Z}}}(\{W_t^{(\alpha,\beta)}\}_{t>0})(f)(n)|&\le C\sum_{m\not\in[n/2,3n/2]}\frac{|f(m)|}{|n-m|}\\
&\le C\left( \frac{1}{n}\sum_{m=0}^{n-1}|f(m)|+\sum_{m=n+1}^\infty\frac{|f(m)|}{m}\right), \quad n\in{\N_0}.
\end{align*}
Here, when $n=0$, the first term in the last sum does not appear. Here, $C$ does not depend on $M$. By using $\ell^p$-boundedness properties of discrete Hardy operators, we deduce that the operator $S_{\{a_j\}_{j\in{\Z}},*,M,glob}^{\{b_j\}_{j\in{\Z}}}(\{W_t^{(\alpha,\beta)}\}_{t>0})$ is bounded from  ${\ell^p(\N_0)}$ into itself, for every $1<p<\infty$. Furthermore, we have that
$$
\sup_{M\in\N}\|S_{\{a_j\}_{j\in{\Z}},*,M,glob}^{\{b_j\}_{j\in{\Z}}}(\{W_t^{(\alpha,\beta)}\}_{t>0}) \|_{{\ell^p(\N_0)}\to {\ell^p(\N_0)}}<\infty, 
$$
for every $1<p<\infty$.

Let $1<p<\infty$. We choose $1<q<p$. $\mathcal{M}_q$ defines a bounded operator from ${\ell^p(\N_0)}$ into itself.

According to {Theorem 3.1}, the operator $S_{\{a_j\}_{j\in{\Z}},(-M,M)}^{\{b_j\}_{j\in{\Z}}}(\{W_t^{(\alpha,\beta)}\}_{t>0})$ is bounded from  ${\ell^p(\N_0)}$ into itself. Moreover, we have that
$$
\sup_{M\in\N}\| S_{\{a_j\}_{j\in{\Z}},(-M,M)}^{\{b_j\}_{j\in{\Z}}}(\{W_t^{(\alpha,\beta)}\}_{t>0})\|_{{\ell^p(\N_0)}\to {\ell^p(\N_0)}}<\infty.
$$
As above, by using {\eqref{eq3.A}} and the $\ell^p$-boundedness properties of discrete Hardy operators, we can deduce that the operator $S_{\{a_j\}_{j\in\Z},(-M,M),glob}^{\{b_j\}_{j\in\Z}}(\{W_t^{(\alpha,\beta)}\}_{t>0})$ is bounded from  ${\ell^p(\N_0)}$ into itself and
$$
\sup_{M\in\N}\| S_{\{a_j\}_{j\in\Z},(-M,M),glob}^{\{b_j\}_{j\in\Z}}(\{W_t^{(\alpha,\beta)}\}_{t>0})\|_{{\ell^p(\N_0)}\to {\ell^p(\N_0)}}<\infty.
$$
Then, $S_{\{a_j\}_{j\in\Z},(-M,M),loc}^{\{b_j\}_{j\in\Z}}(\{W_t^{(\alpha,\beta)}\}_{t>0})$ is bounded from  ${\ell^p(\N_0)}$ into itself and
$$
\sup_{M\in\N}\| S_{\{a_j\}_{j\in\Z},(-M,M),loc}^{\{b_j\}_{j\in\Z}}(\{W_t^{(\alpha,\beta)}\}_{t>0})\|_{{\ell^p(\N_0)}\to {\ell^p(\N_0)}}<\infty.
$$
According to Proposition \ref{prop3.3}, $S_{\{a_j\}_{j\in\Z},M,*,loc}^{\{b_j\}_{j\in\Z}}(\{W_t^{(\alpha,\beta)}\}_{t>0})$ is bounded from  ${\ell^p(\N_0)}$ into itself and
$$
\sup_{M\in\N}\| S_{\{a_j\}_{j\in\Z},M,*,loc}^{\{b_j\}_{j\in\Z}}(\{W_t^{(\alpha,\beta)}\}_{t>0})\|_{{\ell^p(\N_0)}\to {\ell^p(\N_0)}}<\infty.
$$
We conclude that $S_{\{a_j\}_{j\in\Z},M,*}^{\{b_j\}_{j\in\Z}}(\{W_t^{(\alpha,\beta)}\}_{t>0})$ is bounded from  ${\ell^p(\N_0)}$ into itself and
$$
\sup_{M\in\N}\| S_{\{a_j\}_{j\in\Z},M,*}^{\{b_j\}_{j\in\Z}}(\{W_t^{(\alpha,\beta)}\}_{t>0})\|_{{\ell^p(\N_0)}\to {\ell^p(\N_0)}}<\infty.
$$
By taking $M\to+\infty$, it follows that the operator $S_{\{a_j\}_{j\in\Z},M,*}^{\{b_j\}_{j\in\Z}}(\{W_t^{(\alpha,\beta)}\}_{t>0})$ is bounded from  ${\ell^p(\N_0)}$ into itself.

We now apply vector-valued Calder\'on-Zygmund  theory for singular integrals (see \cite{RRT1} and \cite{RRT2}).

We can write
$$
S_{\{a_j\}_{j\in\Z},*}^{\{b_j\}_{j\in\Z}}(\{W_t^{(\alpha,\beta)}\}_{t>0})(f)=\left\| S_{\{a_j\}_{j\in\Z},N}^{\{b_j\}_{j\in\Z}}(\{W_t^{(\alpha,\beta)}\}_{t>0})(f)\right\|_{\ell^\infty(\Z \times\Z)}.
$$
For every $N=(N_1,N_2)$, where $N_1,N_2\in\Z$ and $N_1<N_2$ and $f\in \ell^\infty_c(\Z \times\Z)$, we have that
$$
S_{\{a_j\}_{j\in\Z},N}^{\{b_j\}_{j\in\Z}}(\{W_t^{(\alpha,\beta)}\}_{t>0})(f)(n)=\sum_{m\in{\N_0}}\mathcal{Q}_N^{(\alpha,\beta)}(n,m)f(m),\quad n\in{\N_0},
$$
where
$$
\mathcal{Q}_N^{(\alpha,\beta)}(n,m)=\sum_{j=N_1}^{N_2}b_j\left(K_{a_{j+1}}^{(\alpha,\beta)}(n,m)-K_{a_{j}}^{(\alpha,\beta)}(n,m)\right), \quad n,m\in{\N_0}.
$$

According to { \eqref{eq3.A}} and {\eqref{eq3.B}}, by using \cite[Theorem 2.1]{BCFR} we can prove that the operator $S_{\{a_j\}_{j\in\Z},*}^{\{b_j\}_{j\in\Z}}(\{W_t^{(\alpha,\beta)}\}_{t>0})$ is bounded from $\ell^p({\N_0},w)$ into itself, for every $1<p<\infty$ and $w\in A_p({\N_0})$, and from $\ell^1({\N_0},w)$ into $\ell^{1,\infty}({\N_0},w)$, for every $w\in A_1({\N_0})$.
\edproof

\vspace{0.4cm}

{\bf Declarations}

\thanks{
	The authors are partially supported by grant  PID2019-106093GB-I00 from the Spanish Government. The second author is also supported by the Spanish MINECO
	through Juan de la Cierva fellowship FJC2020-044159-I.} 
    
%
%
%
\bibliographystyle{siam}
\bibliography{references}

\end{document}